\documentclass[oneside, book,12 pt]{article}
\usepackage{amsmath,amssymb,amsfonts,amsthm}
\usepackage{mathrsfs}

\usepackage{pstricks} 
\definecolor{darkblue}{rgb}{0 .2 .5}
\usepackage[colorlinks=true,linkcolor=darkblue,citecolor=darkblue,urlcolor=blue,pdfborder={0 0 0}]{hyperref}
\usepackage[font=sf, labelfont={sf,bf}, margin=1cm]{caption}
\usepackage{enumerate}
\usepackage[mathscr]{eucal}
\usepackage{graphicx}

\usepackage{comment}
\usepackage{fancyhdr}
\usepackage{verbatim}

\usepackage{cleveref}
  \crefname{theorem}{Theorem}{Theorems}
  \crefname{thm}{Theorem}{Theorems}
  \crefname{lemma}{Lemma}{Lemmas}
  \crefname{lem}{Lemma}{Lemmas}
  \crefname{remark}{Remark}{Remarks}
  \crefname{prop}{Proposition}{Propositions}
  \crefname{defn}{Definition}{Definitions}
  \crefname{corollary}{Corollary}{Corollaries}
  \crefname{section}{Section}{Sections}
  \crefname{figure}{Figure}{Figures}

\newtheorem{thm}{Theorem}[section]
\newtheorem{lem}[thm]{Lemma}
\newtheorem{lemma}[thm]{Lemma}
\newtheorem{corollary}[thm]{Corollary}
\newtheorem{prop}[thm]{Proposition}
\newtheorem{defn}{Definition}
\newtheorem{remark}[thm]{Remark}

\numberwithin{equation}{section}

\numberwithin{equation}{section}
\providecommand{\keywords}[1]{\textbf{\textit{Keywords:}} #1}

\numberwithin{equation}{section}
\newcommand{\Tg}{{T}_{\lambda}(n)}
\newcommand{\ve}{\varepsilon}

\def\N{\mathbb N}

\def \C{\mathcal C}

\def \U{\mathcal U}
\def \P{\mathbb P}
\def \E{\mathbb E}

\def \lf{\lfloor}
\def \rf{\rfloor}

\newcounter{mycount}
  
\newenvironment{mylist}{\begin{list}{{\rm (\roman{mycount})}}%
{\usecounter{mycount}\itemsep 0pt}}{\end{list}}
\title{\bf Large unicellular maps in high genus}
 \author{Gourab Ray \footnote{University of British Columbia, 1984
     Mathematics Road, Vancouver, BC, V6T 1Z2}}
  \date{{\small \today}}

\begin{document}
 \maketitle 
\begin{abstract}
 We study the geometry of a random unicellular map which is uniformly
 distributed on the set of all unicellular maps whose
 genus size is proportional to the number of edges. We prove that the
 distance between two uniformly selected vertices of such a map is of order $\log
 n$ and the
 diameter is also of order $\log n$ with high probability. We further
 prove a quantitative version of the result that the map is locally planar with high probability. The main ingredient of the proofs is an exploration procedure which uses a bijection due to Chapuy, Feray and Fusy (\cite{C-perm}).
\end{abstract}

\section{Introduction} \label{sec:main_results}
A \textbf{map} is an embedding of a finite connected graph on a
compact orientable surface viewed up to orientation preserving
homeomorphisms such that the complement of the embedding is an union
of disjoint topological discs. Loops and multiple edges are allowed and our maps are also rooted, that is, an oriented edge is specified as the root. The connected components of the
complement are called faces. The genus of a
map is the genus of the surface on which it is embedded. If a
map has a single face it is called a \textbf{unicellular map}. On a genus $0$ surface, that is, on the sphere,
unicellular maps are classically known as plane
(embedded) trees. Thus unicellular maps can be viewed as generalization
of a plane tree on a higher genus surface.

Suppose $v$ is the number of vertices in a unicellular map of genus
$g$ with $n$ edges. Then Euler's formula yields

\begin{equation}
v-n = 1-2g \label{eq:euler}
\end{equation}
Observe from \cref{eq:euler} that the genus of a unicellular map with
$n$ edges can be at most $n/2$. We are concerned in this paper with
unicellular maps whose genus grows like $\theta n$ for some
constant $0<\theta<1/2$. Specifically, we are interested the geometry
of a typical element among such maps as $n$ becomes large.

\begin{figure}[t]
\centering{\includegraphics[width = .8 \textwidth]{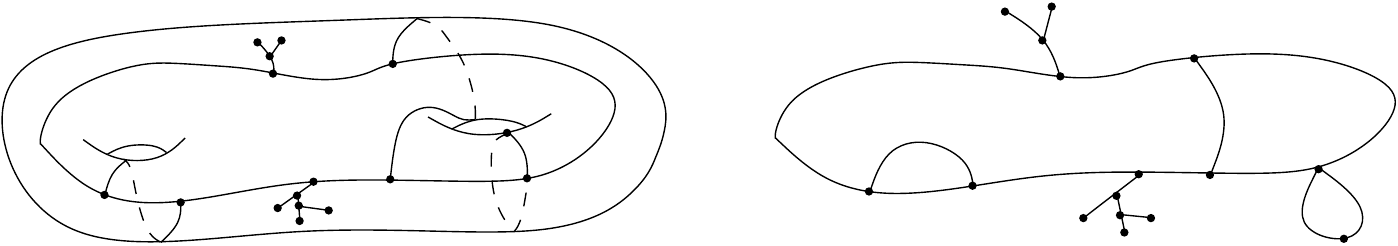} \hfill }
\caption{On the left: a unicellular map of genus $2$. On the right: its underlying graph.}  \label{fig:unicellular}
\end{figure} 

Recall that $\mathcal U_{g,n}$ denotes the set of unicellular maps of genus
$g$ with $n$ edges and let $U_{g,n}$ denote a uniformly picked element
from $\mathcal U_{g,n}$ for integers $g \ge 0$ and $n \ge 1$. For a graph
$G$, let $d^G(.,.)$ denote its graph distance metric. Our first main
result shows that the distance between two uniformly and independently
picked vertices from $U_{g,n}$ is of logarithmic order if $g$ grows
like $\theta n$ for some constant $0 < \theta <1/2$.

\begin{thm}\label{thm:diameter}
 Let $\{g_l,n_l\}_l$ be a sequence in $\N^2$ such that $\{g_l,n_l\}
 \to \{\infty, \infty\}$ and $g_l/n_l \to \theta$
 for some constant
$0<\theta<1/2$. Suppose $V_1$ and $V_2$ are two uniformly and
independently picked vertices from $U_{g_l,n_l}$. Then there exists
constants $0<\ve<C$ (depending only on $\theta$) such that 
\begin{mylist}
\item $\P( d^{U_{g_l,n_l}}(V_1,V_2) >\ve \log n_l) \to 1$ as $l\to \infty$.
\item $\P(d^{U_{g_l,n_l}}(V_1,V_2) >C \log n_l) <c(n_l)^{-3}$ for some $c>0$. 
\end{mylist}

\end{thm}

We remark here that in the course of the proof of part (i) of
\cref{thm:diameter}, a polynomial lower bound on the
rate of convergence will be obtained. But since it is far from being sharp and is not
much more enlightening, we exclude it from the statement of the
Theorem. For part (ii) however, we do provide an upper bound on the
rate. Notice that part (ii) enables us to immediately conclude
that the diameter of $U_{g_l,n_l}$ is also of order $\log n$ with high
probability. For any finite map $G$, let
$diam(G)$ denote the diameter of its underlying graph.

\begin{corollary}
Let $\{g_l,n_l\}_l$ be a sequence in $\N^2$ such that $\{g_l,n_l\}
 \to \{\infty, \infty\}$ and $g_l/n_l \to \theta$
 for some constant
$0<\theta<1/2$. Then there exists constants $\ve>0,C>0$ such that
\[
\P(\ve \log n < diam(U_{g_l,n_l})<C\log n)\to 1
\]
as $n \to \infty$.
\end{corollary}
\begin{proof}
The existence of $\ve>0$ such that $\P( diam(U_{g_l,n_l}) >\ve \log
n) \to 1$ follows directly from \cref{thm:diameter} part (i). For the other
direction, pick the same constant $C$ as in \cref{thm:diameter}. Let
$N$ be the number of pairs of vertices $(v,w)$ in $U_{g_l,n_l}$ where the distance between
them is least $ C \log n$. From part (ii) of \cref{thm:diameter},
$\E(N) <cn_l^{-1}$ for some $c>0$. Hence $\E N$ converges to $0$ as $l
\to \infty$. Consequently, $\P(N >0)$
also converges to $0$ which completes the proof.
\end{proof}

If the genus is fixed to be $0$, that is in the case of plane trees,
the geometry is well understood (see \cite{randomtrees} for a nice
exposition on this topic.) In particular, it can be shown that the
typical distance between two uniformly and independently
picked vertices of a uniform random plane tree with $n$ edges is of
order $\sqrt{n}$. The diameter of such plane trees is also of order
$\sqrt{n}$. These variables when properly rescaled, converge in
distribution to appropriate functionals of the Brownian excursion. This characterization stems from the fact that a plane tree can be viewed as a metric space and the metric if rescaled by $\sqrt{n}$ (up to constants) converges in the Gromov-Hausdorff topology (see \cite{gromov} for precise definitions) to the Brownian continuum random tree (see \cite{CRT1} for more on this.) 
The Benjamini-Schramm limit in the local topology (see \cite{UIPT1,BeSc} for definitions), of the plane tree as the number of edges grow to infinity is also well understood: the limit is a tree with an infinite spine with critical Galton-Watson trees of geometric$(1/2)$ offspring distribution attached on both sides (see \cite{Kestensubdiff} for details.)

Thus \cref{thm:diameter} depicts that the picture is starkly different
if the genus of unicellular maps grow linearly in the number of
vertices. The main idea behind the proof of \cref{thm:diameter} is
that locally, $U_{g,n}$ behaves like a supercritical Galton-Watson
tree, hence the logarithmic order. We believe that the quantity
$d^{U_{g_l,n_l}}(V_1,V_2)$ of \cref{thm:diameter} when rescaled by
$\log n$ should converge to a deterministic constant. Further, we also believe that
the diameter of $U_{g_l,n_l}$ when rescaled by $\log n$ should also
converge to another deterministic constant. This constant obtained from the rescaled
limit of the diameter should be different from the constant obtained as
a rescaled limit of typical distances. The heuristic behind this extra length of the diameter
is the existence of large ``bushes'' of order $\log n$ on the scheme of the
unicellular map
(scheme of a unicellular map is obtained by iteratively deleting all the leaves and then
\textit{erasing} the degree $2$ vertices of the map,) a behaviour
reminiscent of Erdos-Renyi random graphs (see \cite{uni1} for more on
schemes.)

It is worth mentioning here that unicellular maps have appeared frequently in the field of
combinatorics in the past few decades. It is related to representation theory of
symmetric group, permutation factorization, matrix integrals
computation and also the general theory of enumeration of maps. See
the introduction section of \cite{uni1,uni2} for a nice overview and
see \cite{LZ} for connections to other areas of mathematics and
references therein. 

Recall that a quadrangulation (resp. triangulation) is a map where each
face has degree $4$ (resp. $3$). It has been known for
some time that distributional
limits in the local topology of rooted maps (see \cite{BeSc} for
definitions) of uniform triangulations/quadrangulations of the sphere
exists and the limiting measure is popularly known as uniform infinite
planar triangulation/quadrangulation or UIPT/Q in short (see
\cite{UIPT1,UIPT2,Kri05}.)
Our
interest and main motivation for this work is creating hyperbolic analogues
of UIPT/Q. It is believed that uniform
triangulations/quadrangulations of a
surface whose genus is proportional to the number of faces of the map converges in
distribution to a hyperbolic analogue of the UIPT/Q if the
distributional limit is planar, that is, there are no \textit{handles} in the limit. A plausible construction of such a limiting
hyperbolic random quadrangulation, known as {\em stochastic hyperbolic
  infinite quadrangulation} or SHIQ, can be found in \cite{SHIQ}. A
half planar
version of such hyperbolic
maps also arise in \cite{AR13}. It is worth mentioning
here that such limits are expected to hold for any reasonable class of
maps and there is nothing special about quadrangulations or
triangulations. As is the general
strategy in this area, we attempt to attack the problem for quadrangulations using the
bijections between
labelled unicellular maps and quadrangulations of the same genus (see \cite{unibij}.)
Understanding high genus random unicellular maps can be the first step in this
direction. Firstly, understanding whether $U_{g,n}$ is locally planar
with high probability is a question of interest here.

 Tools developed for proving \cref{thm:diameter} also helps us conclude that locally $U_{g,n}$ is in fact planar with high probability which is our next main
 result. In fact, we are also able to quantify up to what distance from the root does $U_{g,n}$ remain planar. This will be made precise in the next theorem. A natural
 question at this point is what is the planar distributional limit of
 $U_{g,n}$ in the local topology. This is investigated in \cite{ACCR13}.

 We now introduce the notion of \textbf{local injectivity radius} of a map.
Since random permutations will play a crucial role in this paper, there will be two notions of cycles floating around:\ one for cycle decomposition of permutations and the
other for maps and graphs. To avoid confusion, we shall refer to a cycle in the
context of graphs as a \textbf{circuit}. A circuit in a planar map is a subset of its vertices and
edges whose image under the embedding is topologically a loop. A circuit is called contractible if its
image
under the embedding on the surface can be contracted to a point. A circuit is called
non-contractible if it is not contractible.

\begin{defn}
The local injectivity radius of a planar map with root vertex $v^*$ is the largest $r$
such that the sub-map formed by all the vertices within graph
distance $r$ from $v^*$ does not contain any non-contractible circuit.
\end{defn}
In the world of Riemannian geometry, injectivity radius around a point $p$ on a
Riemannian manifold refers to the largest $r$ such that the ball of radius $r$
around $p$ is diffeomorphic to an Euclidean ball via the exponential map. This
notion is similar in spirit to what we are seeking in our
situation. Notice however that a circuit in a unicellular map is always non-contractible because it
has a single face. Hence looking for circuits and looking for non-contractible
circuits are equivalent in our situation. 

\begin{thm}\label{thm:inj_radius}
Let $\{g_l,n_l\} \to \{\infty, \infty\}$ and $g_l/n_l \to \theta$ for some constant
$0<\theta<1/2$ as $l \to \infty$. Let $I_{g_l,n_l}$ denote the local injectivity radius of $U_{g_l,n_l}$. Then there exists a constant $\ve > 0$ such that

\[
 \P\left(I_{g_l,n_l} > \varepsilon \log n_l \right) \to 1
\]
 as $l \to \infty$.
\end{thm}

Girth or the circuit of the smallest size of $U_{g,n}$ also deserves some comment. It is possible to conclude
via second moment methods that the girth of $U_{g_l,n_l}$ form a tight
sequence. This shows that there
are small circuits somewhere in the
unicellular map, but they are far away from the root with high probability. 

The main tool for the proofs is a bijection due to Chapuy, Feray and
Fusy (\cite{C-perm}) which gives us a connection between unicellular
maps and certain objects called $C$-decorated trees which preserve the
underlying graph properties (details in \cref{sec:Chapuy_bijection}.)
This bijection provides us a clear roadway for analyzing the
underlying graph of such maps. 

From now on fo simplicity, we shall drop the suffix $l$ in
$\{g_l,n_l\}$, and assume $g$ as a function of $n$ such that $g \to
\infty$ as $n \to \infty$ and $g/n \to \theta$ where $0<\theta <1/2$. The proofs that follow will not be affected by such simplification as one might check. For any
sequence $\{a_n\}$ and $\{b_n\}$ of positive integers, $a_n\sim b_n$ means $a_n/b_n \to
1$. Further $a_n = o(b_n)$ means that $a_n/b_n \to 0$ as $n \to
\infty$ and $a_n = O(b_n)$ means that there exists a universal
constant $C>0$ such that $|a_n| < C|b_n|$. Finally $a_n \asymp b_n$
means there exists positive universal constants $c_1,c_2$ such that $c_1b_n < a_n
< c_2b_n$. In what follows, the constants might vary from step to step
but for simplicity, we shall denote the constants which we do not need
anywhere else by $c$. For a finite set $X$, $|X|$ denotes the
cardinality of $X$.

\paragraph{Overview of the paper:} In \cref{sec:preliminaries} we
gather some useful preliminary results we need. Proofs
and references of some of the results in \cref{sec:preliminaries} are
provided in \cref{sec:proofs_perm,sec:appendix2}. An overview of the strategy of the
proofs of \cref{thm:diameter,thm:inj_radius} is given in
\cref{sec:outline}. Part (ii) of \cref{thm:diameter} along with
\cref{thm:inj_radius} is proved in \cref{sec:lower_diameter}. Part (i) of \cref{thm:diameter} is proved in
\cref{sec:upper_diameter}.

\paragraph{Acknowledgements:} The author is indebted to Omer Angel for
carefully reading the manuscript and providing innumerable suggestions
to make the paper more readable. The author would also like to thank
Guillaume Chapuy, Nicolas Curien, Asaf Nachmias, B\'{a}l\'{a}zs R\'{a}th
and Daniel Valesin for several stimulating discussions. The author also thanks the anonymous referee for several useful comments.




\section{Preliminaries}\label{sec:preliminaries}
In this section, we gather some useful results
which we shall need. 
\subsection{The bijection}\label{sec:Chapuy_bijection} 
Chapuy, F\'{e}ray and Fusy in (\cite{C-perm}) describes a bijection
between unicellular maps and certain objects called $C$-decorated trees. The bijection describes a way to
obtain the underlying graph of $U_{g,n}$ by simply gluing together
vertices of a plane tree in an
appropriate way. This description gives us a simple model
to analyze because plane trees are well understood. In this section we
describe the bijection in \cite{C-perm} and define an even simpler model called
marked trees. The model of marked trees will contain all the
information about the underlying graph of $U_{g,n}$.

For a graph $G$, let $V(G)$ denote the collection of vertices and
$E(G)$ denote the collection of edges of $G$. The subgraph induced by a subset $V' \subseteq V(G)$ of
vertices is a graph $(V',E')$ where $E' \subseteq E(G)$ and for every edge $e \in
E'$, both the vertices incident to $e$ is in $V'$.

A
permutation of order $n$ is a bijective map $\sigma:\{1,2,\ldots,n\}\to \{1,2,\ldots,n\}$. As is classically known,
$\sigma$ can be written as a composition of disjoint cycles. Length of a cycle is the number of elements in the cycle. The cycle type of a permutation is an
unordered list of the lengths of the cycles in the cycle decomposition of the
permutation. A \textbf{cycle-signed} permutation of order $n$ is a permutation of order $n$ where each cycle in its cycle decomposition carries a sign, either $+$ or $-$. 
\begin{defn}[\cite{C-perm}]
A \textbf{$C$-permutation} of order $n$ is a cycle-signed permutation $\sigma$ of order $n$ such that each cycle of $\sigma$ in its cycle decomposition has odd length. The genus of $\sigma$ is defined to be $(n-N)/2$ where $N$ is the number of cycles in the cycle decomposition of $\sigma$.
\end{defn}

\begin{defn} [\cite{C-perm}]
 A \textbf{$C$-decorated tree} on $n$ edges is the pair $(t,\sigma)$ where $t$ is a
rooted plane tree with $n$ edges and $\sigma$ is a $C$-permutation of order $n+1$. The genus of $(t,\sigma)$
is the genus of $\sigma$. 
\end{defn}

The set of all $C$-decorated trees of genus $g$ is denoted by $\mathcal C_{g,n}$. 
 One can canonically order and number the vertices of $t$ from $1$ to
$n+1$. Hence in a
$C$-decorated tree $(t,\sigma)$, the permutation $\sigma$ can be seen
as a permutation on the vertices of the tree $t$. To obtain the
\textbf{underlying graph of a $C$-decorated tree} $(t,\sigma)$, any pair of
vertices $x,y$ whose numbers are in the same cycle of $\sigma$ are glued
together (note that this might create loops and multiple edges.) The underlying graph of $(t,\sigma)$ is the vertex rooted graph obtained
from $(t,\sigma)$ after this gluing procedure. So there are $N$
vertices of the underlying graph of $(t,\sigma)$, each
correspond to a cycle of $\sigma$ (see \cref{fig:bijection}). By
Euler's formula, if the underlying graph of $(t,\sigma)$ is embedded in a surface
such that there is only one face, then the underlying surface must have genus $g$
given by $N=n+1-2g$. 

\begin{figure}[t]
\centering \hspace{3.2em} {\includegraphics[width = .78 \textwidth]{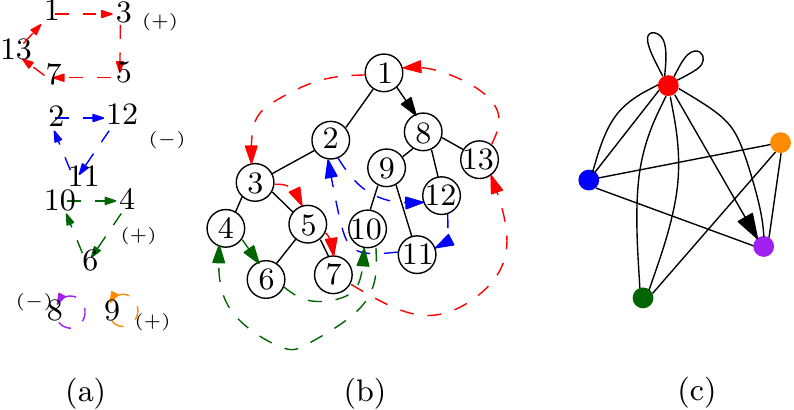} \hfill }
\caption{An illustration of a $C$-decorated tree. (a) A
  $C$-permutation $\sigma$ where each cycle is marked with a different
  color. (b) A plane tree $t$ with the vertices in the same cycle
  of $\sigma$ joined by an arrow of the same color as the cycle. Note
  that vertices numbered $8$ and $9$ are fixed points in the
  $C$-permutation. (c) The underlying graph of the $C$-decorated tree
  $(t,\sigma)$. The root vertex is circled.}  \label{fig:bijection}
\end{figure}

For a set $\mathcal{A}$, let $k\mathcal{A}$ denote $k$ distinct
copies of $\mathcal{A}$. Recall that underlying graph of a unicellular
map is the vertex rooted graph whose embedding is the map.

\begin{thm}(Chapuy, F\'{e}ray, Fusy \cite{C-perm})\label{thm:bijection}
There exists a bijection 
$$2^{n+1}\mathcal{U}_{g,n} \longleftrightarrow
\mathcal{C}_{g,n}.$$
Moreover, the bijection preserves the underlying graph.
\end{thm}

 As promised, we shall now introduce a further simplified model which
 we call \textbf{
marked tree} to
 analyze the underlying graph of $C$-decorated trees. Let $\mathcal P$
denote the set of ordered $N$-tuple of odd positive integers which add up to
$n+1$. 

\begin{defn}
A \textbf{marked tree} with $n$ edges corresponding to an $N$-tuple 
$\lambda=(\lambda_1,\ldots,\lambda_N) \in \mathcal P$ is a pair $(t,m)$ such that $t \in \U_{0,n}$ and
$m:V(t) \to \N$ is a function which takes the value $i$ for exactly $\lambda_i$
vertices of $t$ for all $i = 1,\ldots, N$. The underlying graph of $(t,m)$ is the rooted graph obtained when we merge
together all the vertices of $t$ with the same mark. 
\end{defn}
Given a $\lambda$, let $\mathcal T_\lambda$ be the set of marked trees corresponding to
$\lambda$ and let $T_\lambda$ be a uniformly picked element from it.
Now pick $\lambda$ from $\mathcal P$ according to the following
distribution
\begin{equation}
\P\left(\lambda =(\lambda_1,\lambda_2,\ldots, \lambda_N)\right) = \frac{\prod_{i=1}^N\lambda_i^{-1}}{Z}\label{eq:lambda_dist}
\end{equation}
where $Z =\sum_{\lambda \in \mathcal P}(\prod_{i=1}^N\lambda_i^{-1})$.

\begin{prop}\label{lem:simplify}
Choose $\lambda$ according to the distribution given by
\eqref{eq:lambda_dist}. Then the underlying graph of $U_{g,n}$ and
$T_\lambda$ has the same distribution.
\end{prop}
\begin{proof}
First observe that it is enough to show the following sequence of bijections
$$
2^N\bigcup_{\lambda =(\lambda_1,\ldots,\lambda_N) \in \mathcal P}\prod_{i=1}^N (\lambda_i-1)! \mathcal{T}_{\lambda}(n)
\stackrel{\Psi}{\longleftrightarrow} N!
\C_{g,n}\stackrel{\Phi}{\longleftrightarrow}
2^{n+1}N!\U_{g,n} 
$$
where $\Phi$ and $\Psi$ are bijections which preserve the underlying
graph. This is because for each $\lambda \in \mathcal P$, it is easy
to see that the number
of elements in $\prod_i (\lambda_i-1)! \mathcal{T}_{\lambda}(n)$ is
\mbox{$(n+1)!\prod_{i=1}^N\lambda_i^{-1}$} and given a $\lambda$, the
underlying graph of an uniform element of $\prod_i (\lambda_i-1)!
\mathcal{T}_{\lambda}(n)$ and $\Tg$ has the same distribution.

  Now the existence of bijection $\Phi$ which also preserves the
  underlying graph is guaranteed from \cref{thm:bijection}. For
  $\Psi$, observe that the factor $\prod_{i=1}^N (\lambda_i-1)!$
  comes from the ordering of the elements within the cycle of
  $C$-permutations and the factor $2^N$ comes from the signs
  associated with each cycle of the $C$-permutations. The factor $N!$
  comes from all possible ordering each cycle type of a $C$-permutation
  which is taken into account in the marked trees but not
  $C$-permutations. The details are safely left to the reader. 
\end{proof}

Because of Lemma \ref{lem:simplify} it is enough to look at the
underlying graph of $\Tg$ to prove the Theorems stated in \cref{sec:main_results} where
$\lambda$ is chosen according to the distribution given by
\eqref{eq:lambda_dist}. Our strategy is to show that a typical
$\lambda$ satisfies some ``nice'' conditions (which we will call
condition $(A)$ later), condition on such a $\lambda$ satisfying those
conditions and then work with $\Tg$.

Recall $N=n+1-2g$. Since $g/n \to \theta$ where $0<\theta <1/2$, $n/N
\to (1-2\theta)^{-1}$. Denote $\alpha = (1-2\theta)^{-1}$. Clearly
$\alpha > 1$. The reader should bear in mind that $\alpha$ will remain in the background throughout the rest of the paper.


\subsection{Typical $\lambda$}\label{sec:odd_permutation}
Recall the definition of $\mathcal P$ from
\cref{sec:Chapuy_bijection}. Suppose
$C_0,C_1,C_2,d_1,d_2$ are some positive constants which we will fix later. We say that an element in $\lambda =
(\lambda_1,\lambda_2, \ldots \lambda_N) \in \mathcal P$ satisfies \textbf{condition
$(A)$} if it satisfies
\begin{mylist}
\item $\lambda_{\max} < C_0 \log n$ where $\lambda_{\max}$ is the maximum in the set
$\{\lambda_1,\lambda_2, \ldots \lambda_N\}$. 
\item $C_1 n<\sum_{i=1}^N\lambda_i^2<\sum_{i=1}^N\lambda_i^3 < C_2n$.
\item $d_1n <|i:\lambda_i=1|<d_2n$ 
\end{mylist}
The following Lemma ensures that $\lambda$ satisfies condition $(A)$
with high probability for appropriate choice of the constants. The proof is provided in \cref{sec:proofs_perm}
\begin{lem}\label{cor:conditionA}
Suppose $\lambda$ is chosen according to the distribution given by
\eqref{eq:lambda_dist}. Then there exists constants
$C_0,C_1,C_2 ,d_1,d_2$ depending only upon $\alpha$ such that condition
$(A)$ holds with probability at least $1-cn^{-3}$ for some constant $c>0$. 
\end{lem}

Now we state a Lemma which will be useful later. Given a $\lambda$, we
shall denote by $\P_\lambda$ the conditional measure induced by
$T_\lambda$.

\begin{lem}\label{lem:revealed_cycle}
Fix a tree $t\in \U_{0,n}$ and a $\lambda \in \mathcal P$ satisfying
condition $(A)$. Fix $\mathcal I \subset \{1,2,\ldots,
N\}$ such that $|\mathcal I| < n^{3/4}$. Condition on the event $\mathcal E$ that the plane tree of $\Tg$ is $t$ and $S$ is the set 
of all the vertices in $t$ whose mark belong to $\mathcal I$ where $S$ is some fixed subset of $V(t)$ ($S$ is chosen so that $\mathcal E$ has non-zero probability.) Let $\{v,w,z\}
\subset V(t)\setminus S$ be any set of three distinct vertices in $t$ and $i
\notin \mathcal I$. Then
\begin{align} 
\P_{\lambda}(m(v) = i | \mathcal E) &\sim \lambda_i/n\label{eq:analogue4}\\
 \P_\lambda(m(v) = m(w) |\mathcal E) &\asymp
n^{-1}\label{eq:analogue3}\\
\P_\lambda(m(v) = m(w) = m(z) | \mathcal E) & \asymp n^{-2}\label{eq:analogue2}
\end{align}
\end{lem}
\begin{proof}
Notice that $|S| <C_0n^{3/4} \log n$ because of part (i) of condition $(A)$.
 The proof of \eqref{eq:analogue4} follows from the fact that 
\begin{equation}
 \P_{\lambda}(m(v) = i | \mathcal E) = \frac{(n-|S| - 1)!\lambda_i!}{(n-|S|)!
(\lambda_i-1)!} = \frac{\lambda_i}{n-|S|} \sim \lambda_i/n\nonumber
\end{equation}
since $|S| <C_0n^{3/4}\log n$.

Now we move on to prove \eqref{eq:analogue3}. Conditioned on
$S,t$ the probability that $v$ and $w$ have the same mark $j \notin
\mathcal I$ with $\lambda_j \ge 3$ is 
\begin{equation}
\frac{(n-|S|-2)!\lambda_j!}{(n-|S|)!(\lambda_j-2)!} \sim
\frac{\lambda_j(\lambda_j-1)}{n^2}\nonumber
\end{equation}
All we need to prove is $\sum_{j \notin \mathcal I}
\lambda_j(\lambda_j-1) \asymp n$
which is clear from part (ii) of condition $(A)$ and the fact that
$|\mathcal I | <n^{3/4} $.

Proof of \cref{eq:analogue2} is very similar to that of \cref{eq:analogue3} and
is left to the reader.  
\end{proof}

\subsection{Large deviation estimates on random trees}\label{sec:Galton_Watson}
\subsubsection{Galton-Watson trees}
 A Galton-Watson tree, roughly speaking, is the family tree of a Galton-Watson
process which is also sometimes referred to as a branching process in the
literature. These are well studied in the past and goes far back to
the work of Harris (\cite{Harris}). A fine comprehensive coverage about
branching processes can be found in  \cite{Athreya2}. Given a
Galton-Watson tree, we denote by $\xi$ the offspring distribution. Let
$\P(\xi = k) = p_k$ for $k \ge 1$. Let
$Z_r$ be the number of vertices at generation $r$ of the tree. We shall
also assume
\begin{itemize}
\item $p_0 + p_1 <1$
\item $\E(e^{\lambda \xi}) < \infty$ for small enough $\lambda >0$.
\end{itemize}

We need the following lower deviation estimate. The proof essentially
follows from a result in \cite{Athreya2} and is provided in \cref{sec:appendix2}.
\begin{lem}\label{lem:lowerdev}
Suppose $\E \xi = \mu>1$ and the distribution of $\xi$ satisfies the assumptions as above. For any constant $\gamma$ such that $ 1<\gamma
< \mu$, for all $r \ge 1$
\[
 \P(Z_r \le \gamma^r) < c\exp(-c'r) + \P(Z_r = 0)
\]
for some positive constants $c,c'$.
\end{lem}
\subsubsection{Random plane trees}  \label{sec: Random
Plane Trees}
 A random plane tree with $n$ edges is a uniformly picked ordered tree
 with $n$ edges (see \cite{randomtrees} for a formal treatment.) In other words a random plane
 tree with $n$ edges is nothing but $U_{0,n}$ as per our notation. We shall need the following large deviation result for the lower bounds
and upper bounds on
the diameter of $U_{0,n}$. This follows from Theorem $1.2$ of \cite{tail_bound} and the discussion in
Section $1.1$ of \cite{tail_bound}. 
 \begin{lem}\label{lem:diameter_tree}
 For any $x >0$, \begin{mylist}
                  \item $P(Diam(U_{0,n}) \le x) < c\exp(-c_1(n-2)/x^2)$
\item $P(Diam(U_{0,n}) > x) < c\exp(-c_1x^2/n)$
                 \end{mylist}
where $c>0$ and $c_1 > 0$ are constants.
\end{lem}
We shall also need some estimate of local volume growth in random plane trees. For this purpose, let us define for an integer $r \ge 1$,
\[
  M_{r} = \max_{v \in V(U_{0,n})}|B_r(v)|
\]
where $B_r(v)$ denotes the ball of radius $r$ around $v$ in the graph
distance metric of $U_{0,n}$.
In other words, $M_r$ is the maximum over $v$ of the volume of the ball of radius $r$
around a vertex $v$ in $U_{0,n}$. It is well known that typically, the ball of radius $r$ in
$U_{0,n}$ grows like $r^2$. The following Lemma states that $M_r$ is
not much larger than $r^2$ with high probability. Proof is provided in \cref{sec:appendix2}.

\begin{lem}\label{lem:max_ring}
Fix $j\ge 1$ and $r=r(n)$ is a sequence of integers such that $1 \le r(n) \le n$. Then
there exists a constant $c > 0$ such that
\begin{equation}
 \P(M_r > r^2 \log^2n) <\exp(-c \log^{2} n )\nonumber
\end{equation}
\end{lem}

\section{Proof outline}\label{sec:outline}
In this section we describe the heuristics of the proofs of \cref{thm:diameter,thm:inj_radius}.

Let us describe an exploration process on a given marked tree
starting from any vertex $v$ in the plane tree. This process will
describe an increasing sequence of subsets of vertices which we will call
the set of revealed vertices. In the first step, we 
reveal all the vertices with the same mark as $v$. Then we explore the
set of revealed vertices one by one. At each step when we explore a
vertex, we reveal all its neighbours and also reveal all the vertices
which share a mark with one of the neighbours. If a neighbour has already been
revealed, we ignore it. We then explore the unexplored vertices and continue.

We can associate a branching process with this exploration process where the number of vertices revealed while exploring a vertex can be thought of as the offsprings of the vertex. It is well known that the degree of any uniformly picked vertex in $U_{0,n}$ is roughly distributed as a geometric$(1/2)$ variable and we can expect such behaviour of the degree as long as the number of vertices revealed by the exploration is small compared to the size of the tree. Now the
expected number of vertices with the same mark as a vertex is roughly
a constant strictly larger
than $1$ because of part (ii) of condition $(A)$. Hence the associated
branching process will have expected number of offsprings a constant
which is strictly
larger than $1$. Thus we can stochastically dominate this branching
process both from above and below by supercritical Galton-Watson processes which will account for the logarithmic order of typical distances. 

Once we have such a domination, observe that the vertices at
distance at most $r$ from the root in the underlying graph of the marked tree
is approximately the vertices in the ball of radius $r$ around the
root in a supercritical
Galton-Watson tree. Hence by virtue of the fact that supercritical
Galton-Watson trees have roughly exponential growth, we can conclude
that the number of vertices at a distance at most $\ve \log n$ from the
root in the underlying graph of the marked tree is $\ll \sqrt{n}$ if
$\ve>0$ is small enough. Hence note that to have a circuit within distance
$\ve \log n$ in the underlying graph of the marked tree, two of the
vertices which are revealed within $\ll
\sqrt{n}$ many steps must be close in the plane tree. But observe that
the distribution of the revealed vertices is roughly a uniform sample
from the set of vertices in the tree up to the step when at most roughly $\sqrt{n}$ many
vertices are revealed.
Hence the probability of
revealing two vertices which are close in the plane tree up to roughly $\sqrt{n}$ many steps is small
because of the birthday paradox argument. This argument shows that the local
injectivity radius is at least $\ve \log n$ for some small enough $\ve>0$.

The rest of the paper is the exercise of making these heuristics precise.


\section{Lower Bound and Injectivity
  radius}\label{sec:lower_diameter}
Recall condition $(A)$ as described in the begininning of \cref{sec:odd_permutation}. Pick a $\lambda$ satisfying condition $(A)$. 
Recall that $\Tg$ denotes a uniformly picked element from $\mathcal
T_\lambda(n)$. Throughout this section we shall fix a $\lambda$ satisfying condition $(A)$ and work with $\Tg$. Also recall that $\Tg = (U_{0,n},M)$ where $U_{0,n}$ is a
uniformly picked plane tree with $n$ edges and $M$ is a uniformly picked marking
function corresponding to $\lambda$ which is independent of $U_{0,n}$. Let $d_\lambda(.,.)$ denote the
graph distance metric in the underlying graph of $\Tg$.
In this section we prove the following Theorem. 
  \begin{thm} \label{thm:lower_marked_tree}
Fix a $\lambda$ satisfying condition $(A)$. Suppose $x$ and $y$
are two uniformly and independently picked numbers from
$\{1,2,\ldots,N\}$ and $V_x$ and $V_y$ are the vertices in the underlying graph of $\Tg$
corresponding to the marks $x$ and $y$ respectively. Then there exists
 a constant $\ve>0$ such that 
\[
 \P_{\lambda}(d_{\lambda}(V_x,V_y) < \ve \log n) \to 0
\]
as $n \to \infty$.
\end{thm}
\begin{proof}[Proof of \cref{thm:diameter} part (i)]
Follows from \cref{thm:lower_marked_tree} along with
\cref{lem:simplify,cor:conditionA}.
\end{proof}

 As a by-product of the
proof of \cref{thm:lower_marked_tree}, we also obtain the proof of \cref{thm:inj_radius} in this section.

Note that for any finite graph, if the volume growth around a typical vertex is
small, then the distance between two typical vertices is large. Thus to prove \cref{thm:lower_marked_tree}, we aim to prove an
upper bound on volume growth around a typical vertex. 
Note that with high probability the
maximum degree in $U_{0,n}$ is logarithmic and $\lambda_{\max}$ is also
logarithmic (via condition $(A)$ part (i) and \cref{lem:diameter_tree}.)\ Hence
it is easy to see using the idea described in \cref{sec:outline} that
the typical distance is at least $\ve \log n/ \log \log n$ with high
probability if $\ve > 0$ is small enough. This is
enough, as is heuristically explained in \cref{sec:outline}, to ensure that the injectivity radius of
$U_g(n)$ is at least $\ve \log n/\log \log n$ with high probability
for small enough constant $\ve >0$. The rest of this
section is devoted to the task of getting rid of the $\log \log n $
factor. This is done by ensuring that while performing the exploration
process for reasonably small number of steps, we do not reveal
vertices of high degree with high probability.

Given a marked tree $(t,m)$, we shall define a nested sequence
$R_0 \subseteq R_1 \subseteq
R_2 \subseteq \ldots$ of subgraphs of $(t,m)$ where $R_k$ will be the called the \textbf{
  subgraph revealed} and the vertices in $R_k$ will be called the \textbf{
  vertices revealed} at the $k$th step of the exploration process. We
will also think of the number of steps as the amount of time the
exploration process has evolved. There will be two states of
the vertices of $R_k$:~\textbf{active} and \textbf{neutral}.\ Along with
$\{R_k\}$, we will define another nested sequence $E_0 \subseteq E_1 \subseteq
E_2 \subseteq \ldots$. In the first step, $R_0 = E_0$ will
be a set of vertices with the same mark and hence $E_0$ will correspond to a single vertex in the
underlying graph of $(t,m)$. The subgraph of the
underlying graph of $(t,m)$ formed by gluing together vertices with
the same mark in $E_r$ will be the ball of radius
$r$ around the vertex corresponding to $E_0$ in the underlying graph
of $(t,m)$. The process will have rounds and during round $i$, we
shall reveal the vertices which correspond to vertices at distance
exactly $i$ from the vertex
corresponding to $E_0$ in the underlying graph of
$(t,m)$. Define $\tau_0=0$ and we now define $\tau_r$ which will denote the time of completion of
the $r$th round for $r \ge 1$. Let $N_r = E_r \setminus E_{r-1}$.
Inductively, having defined $N_r$, we continue
to explore every vertex in $N_{r}$ in some predetermined order and
$\tau_{r+1}$ is the step when we finish exploring $N_r$. For a vertex $v$,
 $mark(v)$ denotes the set of marked vertices with the same mark as
 that of $v$. For a vertex set $S$, $mark(S) = \cup_{v \in
   S}mark(v)$. We now give a  rigorous algorithm for the exploration process. 

\begin{description}
\item[\underline{Exploration process I}]\hspace*{1mm}
\begin{mylist}
\item \textbf{Starting rule:} Pick a number $x$ uniformly at random
  from the set of marks $\{1,2,\ldots , N\}$ and
let $E_0 = R_0 = mark(x)$. Declare all the vertices in $mark(x)$ to be
active. Also set $\tau_0 = 0$.

\item \textbf{Growth rule:} \begin{enumerate}
                             
\item For some $r \ge 1$, suppose we have defined the nested subset of vertices of  $E_0\subseteq \ldots \subseteq E_r$ such that $N_{r}: =E_r\setminus
  E_{r-1}$ is the set of active vertices in $E_r$. Suppose we have
  defined the increasing sequence of times $\tau_0\le\ldots\le \tau_r$
  and the nested sequence of subgraphs $R_0\subseteq R_1\subseteq
  \ldots\subseteq R_{\tau_r}$ such that $R_{\tau_r} = E_r$. The number
  $r$ denotes the number of \textbf{rounds} completed in the
  exploration process at time $\tau_r$.

\item Order the vertices of $N_r$ in some
  arbitrary order. Now we explore the first vertex $v$ in the ordering of $N_r$. Let
$S_v$ denote all the neighbours of $v$ in $t$ which do not belong to
$R_{\tau_r}$. Suppose $S_v$ has $l$ vertices $\{v_1,v_2,\ldots,v_l\} $ which are
  ordered in an arbitrary way. For $1\le j \le l$, at step $\tau_r+j$, define
  $R_{\tau_r+j}$ to be the subgraph induced by $V(R_{\tau_r+j-1}) \cup mark(v_j) $. At step $\tau_r+l$ we finish exploring $v$. Define all the vertices in $R_{\tau_r+l} \setminus R_{\tau_r}$ to be active and declare $v$ to be neutral. Then we move on to the next vertex in $N_r$ and continue.
\item  Suppose we have finished exploring a vertex of $N_r$ in step
  $k$ and obtained $R_k$. If there are no more vertices left in $N_r$,
  define $k = \tau_{r+1}$ and $E_{r+1} = R_{\tau_{r+1}}$.\ Declare round $r+1$ is completed and go to step $1$.
\item Otherwise, we move on to the next vertex $v'$ in $N_r$ according to the predescribed order. Let $S_{v'} = \{v_1,v_2,\ldots,v_{l'}\}$ be the neighbours of $v'$ which do not belong to $R_k$. For $1\le j \le l'$, at step $k+j$, define
  $R_{k+j}$ to be the subgraph induced by $V(R_{k+j-1}) \cup mark(v_j) $. Define all the vertices in $R_{k+l'} \setminus R_{k}$ to be active and declare $v'$ to be neutral. Now go back to step $3$. 
 \end{enumerate}


\item \textbf{Threshold rule:} We stop if the number of steps
  exceeds $n^{1/10}$ or the number of rounds exceeds $\log n$. Let $\delta$ be the step number when we stop the exploration process.
\end{mylist}
 
\end{description}

Recall that $V_x$ denotes the vertex in the underlying graph of $\Tg$
corresponding to the mark $x$.\ The following proposition
is clear from
the description of the exploration process and is left to the reader
to verify.
\begin{prop}\label{prop:exploration_distance}
For every $j \ge 1$, all the vertices with the same mark in $E_j \setminus E_{j-1}$
when glued together form all the vertices at a distance exactly $j$ from
$V_{x}$ in the underlying graph of $(t,m)$.
\end{prop}

In step $0$, define $mark(x)$ to be the \textbf{seeds} revealed in
step $0$. At any step, if we reveal $mark(z)$ for some vertex $z$,
then $mark(z) \setminus z$ is called the seeds revealed at that
step. The nomenclature seed comes from the fact that a
seed gives rise to a new connected component in the
revealed subgraph unless it is a neighbour of one of the revealed
subgraph components. However we shall see that the probability of the
latter event is small and typically every connected component has one
unique seed from which it ``starts to grow''.

Now suppose we perform the exploration process on $\Tg = (U_{0,n},M)$
where recall that $M$ is a uniformly random marking function which is compatible with $\lambda$ on the set of vertices of
$U_{0,n}$ and is independent of the tree $U_{0,n}$. Let $\mathcal F_k$ be the sigma field generated by
$R_0,R_1,R_2,\ldots, R_k$.

The aim is to control the growth of $R_k$ and to that end, we need to
control the size of $mark(S_v)$ while exploring the vertex $v$
conditioned up to what we have revealed up to the previous step. It
turns out that it will be more convenient to condition on a subtree
which is closely related to the connected
tree spanned by the vertices revealed.

\begin{figure}[t]
\centerline{\hspace{5em}{\includegraphics[width=0.75 \textwidth]{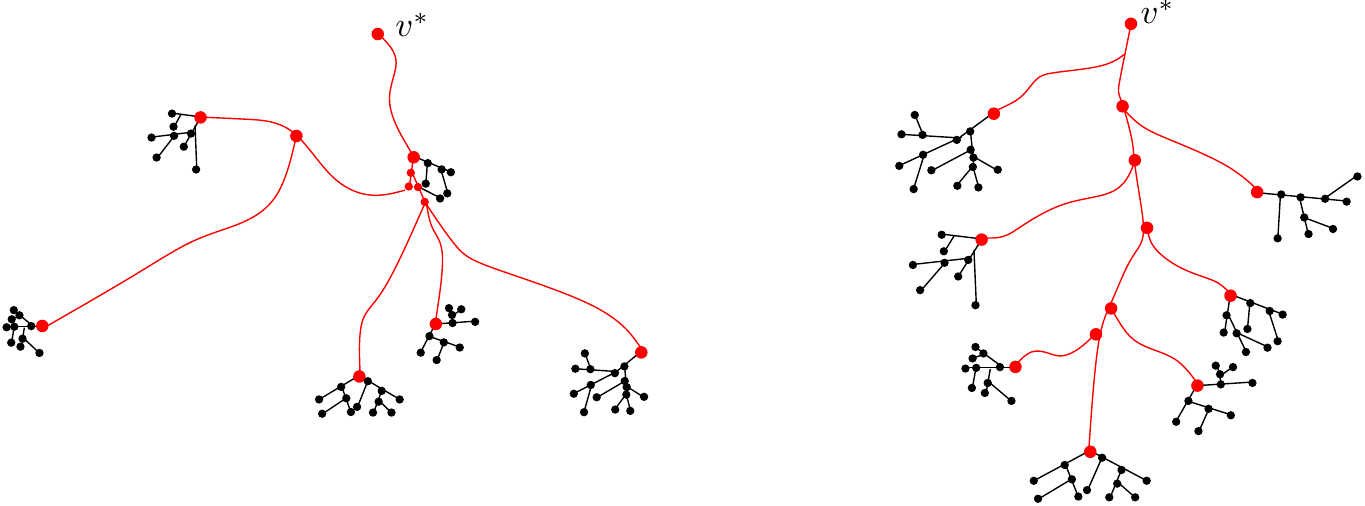} \hfill }} 
\caption{The web is denoted by the red paths. On the left: a general web structure. Apriori the web structure might be very complicated. Many
  paths in the web might pass through the same vertex as is depicted
  here. On the right: A typical web structure} \label{fig:web}
\end{figure} 

\begin{defn}
The \textbf{web} corresponding
to $R_k$ is defined to be the union 
of the unique paths joining the root vertex $v^*$ and the vertices closest to $v^*$ in each of the connected components of $R_k$ 
including the vertices at which the paths intersect $R_k$. The web corresponding to 
$R_k$ is denoted by $P_{R_k}$.
 \end{defn}

 As mentioned before, the idea is to condition on the web. Observe that after
 removing the web from $U_{0,n}$ at any step, we are left with a uniformly distributed forest
 with appropriate number of edges and trees. What stands in our way is that in general the
 web corresponding to a subtree might be very complicated (see
 \cref{fig:web}). The paths joining the root and several components
 might ``go through'' the same component. Hence conditioned on
 the web, a vertex might apriori have arbitrarily many of its neighbours belonging
 to the web.~To show that this does not happen with high probability
 we need the following definitions.

For any vertex $u$ in $t$, the
\textbf{ancestors} of $u$ are the vertices in $t$ along the unique
path joining $u$ and the root vertex $v^*$. For any two vertices $u,v$ in $t$
let $u \wedge v$ denote the common ancestor of $u$ and $v$ which is
farthest from the root vertex $v^*$ in $t$. Let
$$C(u,v) = d^t(u\wedge v ,\{u,v,v^*\})$$
A pair of
vertices $(u,v) $ is called a \textbf{ bad pair} if $C(u,v) < \log^2n$
(see \cref{fig:badpair}.)
\begin{figure}[t]
\centerline{\hspace{11em}{\includegraphics[width=0.5 \textwidth]{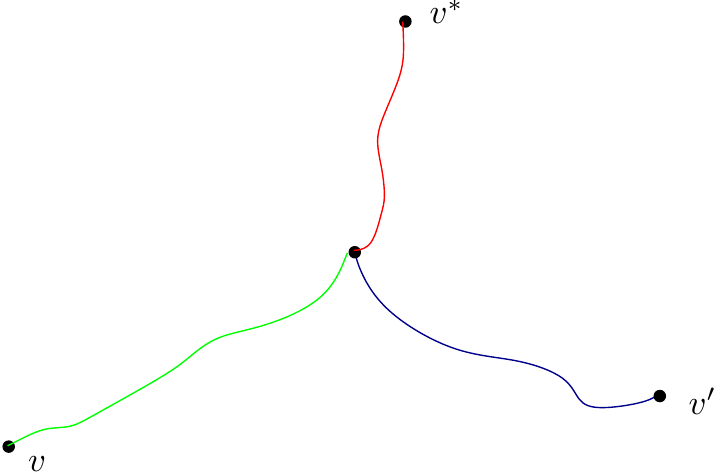} \hfill }} 
\caption{$v^*$ denotes the root vertex. $(v,v')$ is a bad pair if
  either of red, green or blue part has at most $\log^2n$ many vertices.} \label{fig:badpair}
\end{figure}

Recall that we reveal some set of seeds (possibly empty) at each step of the exploration process. Suppose we uniformly order the seeds revealed at each step
and then concatenate them in the order in which they are revealed. More formally, Let $(s_{i_0},s_{i_1},\ldots, s_{i_{k_i}} )$ be the set of seeds revealed in step $i$ ordered in uniform random order. Let $S = (s_{1_0},s_{1_1},\ldots, s_{1_{k_1}},  \ldots, s_{\delta_1},\ldots, s_{\delta_{k_\delta}})$. To simplify notation, let us denote $S=(S_0,S_1,\ldots,S_{\delta'})$ where $\delta'+1$ counts the number of seeds revealed up to step $\delta$. The reason
for such ordering is technical and will be clearer later in the proof
of \cref{lem:tot_var}. 

\begin{lem}\label{lem:web_constraint}
 If $S$ does not contain a bad pair then each connected component of $R_\delta$ contains an unique seed and
 the web $P_{R_\delta}$ intersects each connected component of
 $R_\delta$ at most at one vertex. 
\end{lem}

 \begin{remark}

 In \cref{lem:death}, we shall prove that the
 probability of $S$ containing a bad pair goes to $0$ as $n \to
 \infty$. This and \cref{lem:web_constraint} shows that for large $n$,
 the typical structure of the web is like the right hand figure of 
\cref{fig:web}.
\end{remark}

\begin{proof}
Clearly, every connected component of $R_\delta$ must contain at least one seed.
Also note that every connected component of $R_\delta$ has 
diameter at most $2 \log n$ because of the threshold rule. Since the
distance between any pair of seeds in $R_\delta$ is at least $
\log^2n$ if $S$ do not contain a bad pair, each component must contain a unique seed.

Suppose at any arbitrary step there is a connected component $C$ which intersects the web in
more than two vertices. Then there must exist a component $C'$ such that the path
of the web joining the root and $C'$ intersects $C$ in more than one
vertex. This implies that the (unique) seeds of $C$ and $C'$ form a bad pair since the diameter of both $C$ and $C'$ are at most $2\log n$. 
\end{proof}

Now we want to prove that with high probability, $S$ do not contain a
bad pair. Observe that the distribution of the set of seeds
revealed is very close to a uniformly sampled set of
vertices without replacement from the set of vertices of the tree
as long as $R_{\delta} \ll \sqrt{n}$, because of the same effect as the
birthday paradox. We quantify this statement and further show that an
i.i.d. sample of size $\delta'$ from the set of vertices do not
contain a bad pair with high probability.

 We first show that the cardinality of the set $R_\delta$ cannot be too large with high probability. 
\begin{lem}\label{lem:size}
 $R_\delta=O(n^{1/10}\log n)$
\end{lem}
\begin{proof}
 At each step at most $\lambda_{\max}$ many vertices are revealed and
 $\lambda_{\max} = O(\log n)$ via condition $(A)$.
\end{proof}

 Given $S$, let $\widetilde{S}
 = \{\widetilde{S}_0,\widetilde{S}_1,\ldots, \widetilde{S}_{\delta'}\}$ be an i.i.d. sample of uniformly
 picked vertices from $U_{0,n}$. First we need the following technical Lemma.

\begin{lem}\label{lem:technical}
Suppose $a = a(n)$ and $b = b(n)$ are sequences of positive integers such that
$(a+b)^2 = o(n)$. Then for large enough $n$,
\[
n^b\left|\frac{1}{b!}\dbinom{n-a}{b}^{-1} - \frac{1}{n^b}\right|  < 4\left(\frac{(a+b)b}{n}\right)
\]
\end{lem}
\begin{proof}
Observe that
\begin{align}
\frac{1}{b!}\dbinom{n-a}{b}^{-1} & = \frac{1}{(n-a-b+1)\ldots(n-a)}\nonumber\\
                                & =
                                \frac{1}{n^b}\prod_{j=1}^b\left(1+\frac{a+b-j}{n-(a+b-j)}\right)\label{eq:technical}\\
                                & <  \frac{1}{n^b}\prod_{j=1}^b(1+2(a+b)/n))\nonumber\\
                                & < \frac{1}{n^b}\exp\left(\frac{2(a+b)b}{n}\right)\nonumber\\
                                & = \frac{1}{n^b}\left(1+ \frac{2(a+b)b}{n}+o\left(\frac{(a+b)b}{n}\right)\right)\nonumber\\
                                & < \frac{1}{n^b} \left(1+4\left(\frac{(a+b)b}{n}\right)\right)\nonumber
\end{align}
where the third inequality follows because $n-(a+b)>n/2$ for
large enough $n$ and $a+b-j<a+b$. The second last equality follows since $b(a+b) = o(n)$ via the hypothesis. The other direction follows from the fact that the expression in the right hand side of \cref{eq:technical} is larger than $1/n^b$.
\end{proof}
For random vectors $X,Y$ let $d_{TV}(X,Y)$ denote the total
variation distance between the measures induced by $X$ and $Y$.

\begin{lem}\label{lem:tot_var}
\[
d_{TV}(S,\widetilde{S})  <4n^{-2/3}
\]
\end{lem}
\begin{proof}
First note that $|S| < |R_\delta| < n^{1/9}$ from \cref{lem:size}.
Let $(S_1,S_2,\ldots S_{d})$ be the ordered set of seeds revealed in the first
step after uniform ordering. Then 
\begin{equation}
d_{TV}((S_1,\ldots, S_d), (\widetilde{S}_1,\ldots,
\widetilde{S}_d)) <n^d \left|\frac{1}{d!}\dbinom{n}{d}^{-1} -
  \frac{1}{n^d}\right| <4n^{-7/9}\label{eq:tot_var1}
\end{equation}
where the factor $n^{d}$ in the first inequality of
\eqref{eq:tot_var1} comes from the definition of total variation distance and the fact that there are $n^d$ many $d$-tuple of vertices
and the second inequality of \eqref{eq:tot_var1} follows from
\cref{lem:technical} and the fact that $d<|S| < n^{1/9}$. We will
now proceed by induction on the number of steps.
Suppose up to step $t$, $(S_1,\ldots,S_m)$ is the ordered set of seeds
revealed. Assume 
\begin{equation}
d_{TV}((S_1,\ldots, S_m), (\widetilde{S}_1,\ldots,
\widetilde{S}_m)) <4mn^{-7/9} 
\end{equation}
Recall $\mathcal F_t = \sigma(R_0,\ldots, R_t)$. Now
suppose we reveal $S_{m+1},\ldots,S_{m+L}$ in the $t+1$th
step where $L$ is random depending upon the number of seeds revealed in the $t+1$th step. Observe that to finish the proof of the lemma, it is enough to prove that the total variation distance between the measure induced by $(S_{m+1},\ldots,S_{m+L})$ conditioned on $\mathcal F_t$ and $(\widetilde{S}_{m+1},\ldots,
\widetilde{S}_{m+L})$ (call this distance $\Delta$) is at most $4n^{-7/9}$. This is because using induction hypothesis and $\Delta < 4n^{-7/9}$, we have the following inequality
\begin{multline}
d_{TV}((S_1,\ldots, S_{m+L}), (\widetilde{S}_1,\ldots,
\widetilde{S}_{m+L})) \\< d_{TV}((S_1,\ldots, S_{m}), (\widetilde{S}_1,\ldots,
\widetilde{S}_{m}))+4n^{-7/9} <4(m+1)n^{-7/9} \label{eq:tot_var2}.
\end{multline}
Thus \eqref{eq:tot_var2} along with induction implies $d_{TV}(S,\widetilde{S}) <4n^{1/9}n^{-7/9} <4n^{-2/3}$ since
$\delta' < n^{1/9}$. 

Let $\mathcal{F}_t'$ be the sigma field induced by $\mathcal F_t$ and the mark revealed in step $t+1$.
To prove $\Delta < 4n^{-7/9}$, note that it is enough to prove that the total variation distance between the measure induced by $S_{m+1},\ldots,S_{m+L}$ conditioned on $\mathcal F_t'$ and $(\widetilde{S}_{m+1},\ldots,
\widetilde{S}_{m+L})$ (call it $\Delta'$) is at most $4n^{-7/9}$. But if $l$ many seeds are revealed in step $t+1$ (note $l$ only depends on the mark revealed) then a calculation similar to \eqref{eq:tot_var1} shows that
\begin{equation}
 \Delta' < n^l \left|\frac{1}{l!}\dbinom{n-|R_t|-1}{l}^{-1} -
  \frac{1}{n^l}\right| <4n^{-7/9} \nonumber
\end{equation}
where the last inequality above again follows from \cref{lem:technical}.
The proof is now complete.
\end{proof}


We next show, that the probability of obtaining a bad pair of vertices in the collection of vertices $\widetilde{S}$ is small.
\begin{lem}\label{lem:ind_bad}

\[
 \P_\lambda(\widetilde{S} \text{ contains a bad pair }) = O(n^{-1/10})
\]
\end{lem}
\begin{proof}
 Let $(V,W)$ denote a pair of vertices uniformly and independently picked from the set of vertices of $U_{0,n}$.
Let $P$ be the path joining the root vertex and $V$. Let $A$ be the
event that the unique path joining $W$ and $P$ intersects $P$ at a
vertex which is within distance $\log^2n$ from the root vertex or
$V$. Since $V$ and $W$ have the same distribution and since there are
at most $n^{2/9}$ pairs of vertices in $\widetilde{S}$, it is enough
to prove $\P_\lambda(A) =  O(n^{-1/3}\log^2n)$

Recall the notation $M_r$ of \cref{lem:max_ring}: $M_r$ is the maximum
over all vertices $v$ in $U_{0,n}$ of the volume of the ball of radius
$r$ around $v$. Let $|P|$ denote
the number of vertices in $P$. Consider the event $E =
\{M_{\lf n^{1/3}\rf} < n^{2/3}\log^2n\}$. On $E$, the
probability of $\{|P| <n^{1/3}\}$ is $O(n^{-1/3}\log^2n)$. Since the
probability of the complement of $E$ is $O(\exp(-c\log^2n))$ for some
constant $c>0$
because of \cref{lem:max_ring}, it is enough to prove the bound for the
probability of $A$ on \mbox{$|P| > n^{1/3}$}.

Condition on $P$ to have $k$ edges where $ k>n^{1/3}$. Observe that
the distribution of $U_{0,n} \setminus P$ is given by an uniformly
picked of rooted forests with $\sigma = 2k+1$ trees and $n-k$
edges. Hence if we pick another uniformly distributed vertex $W$
independent of everything else, the unique path joining $W$ and $P$
intersects $P$ at each vertex with equal probability. Hence the
probability that the unique path joining $W$ and $P$ intersects $P$ at
a vertex which is at a distance within $\log^2n$ from the root or $V$
is $O(n^{-1/3}\log^2n)$ by union bound. This completes the proof.
\end{proof}

\begin{lem}\label{lem:death}
\[
\P_{\lambda}(S \text{ contains a bad pair}) = O(n^{-1/10})
\]
\end{lem}
   \begin{proof}
   Using \cref{lem:tot_var,lem:ind_bad}, the proof follows. 
   \end{proof}

We will now exploit the special structure of the web on the event that
$S$ do not contain a bad pair to dominate the
degree of the explored vertex by a suitable random variable of finite
expectation for all large $n$. To this end, we need some enumeration
results for forests. Note that the forests we consider here are rooted and ordered. Let $\Phi_{\sigma,e}$ denote the number of forests with $\sigma$ trees and $e$ edges. It is well known (see for example, Lemma 3 in \cite{betti2}) that 
\begin{equation}
     \Phi_{\sigma,e} = \frac{\sigma}{2e+\sigma}\dbinom{2e+\sigma}{e}\label{eq:forest}
\end{equation}
We shall need the following estimate. The proof is postponed for later.

\begin{lemma}\label{lem:tree_estimate}
Suppose $e$ is a positive integer such that $e < n$. Suppose $d_0,d_1$
denote the degree of the roots of two trees of a uniformly
picked forest with $n-e$ edges and $\sigma$ trees. Let $j \le n-e$. Then
\[
 \max\{\P(d_0+d_1 = j), \P(d_0 =j)\}< \frac{4j(j+1)	}{2^{j}}
\]

\end{lemma}
We shall now show the degree of an explored vertex at any step of the
exploration process can be dominated by a suitable variable of
finite expectation which do not depend upon $n$ or the step
number. Recall that while exploring $v$ we spend several steps of the
exploration process which depends on the number of neighbours of $v$
which have not been revealed before.

In the following \cref{lem:coupling_degree,lem:dom_lower}, we assume $v_{k+1}$ is the vertex we \textit{start exploring} in the $(k+1)$th step of the exploration process.
\begin{lem}\label{lem:coupling_degree}
The distribution of the degree of $v_{k+1}$ conditioned on $R_k$ such that $R_k$ do not contain a bad pair is stochastically
dominated by a variable $X$ where
$\E X < \infty$ and the distribution of $X$ do not depend on $n$ or $k$.
\end{lem}
\begin{proof}
 Consider the conditional distribution of the degree of $v_{k+1}$ conditioned on $R_k$ as well as $P_{R_k}$.  Without loss of generality assume $P_{R_k}$ do not contain $n$ edges for then the Lemma is trivial. Note that $P_{R_k}$ cannot
intersect a connected component of $R_k$ at more than one vertex because of \cref{lem:web_constraint}. Suppose $e<n$ is the number of edges in
the subgraph $P_{R_k} \cup R_k$. It is easy to see that the
distribution of \mbox{$U_{0,n} \setminus (P_{R_k} \cup R_k)$} is a uniformly
picked element from the set of forests with $\sigma$ trees and
$n-e$ edges for some number $\sigma$. If $v_{k+1}$ is not an isolated vertex in $R_k$ (that is there is an edge in $R_k$ incident to $v_{k+1}$), the degree of $v$ is at most $2$ plus the sum of the degrees of
the root vertices of two trees in a uniformly distributed forest of
$\sigma$ trees and $n-e$ edges. If $v_{k+1}$ is an isolated vertex, the
degree of $v_{k+1}$ is $1$ plus the degree of the root of a tree in a
uniform forest of $\sigma $ trees and $n-e$ edges. Now we can use the
bound obtained in \cref{lem:tree_estimate} and observe that the
bound do not depend on the conditioning of the web $P_{R_k}$.
It is easy now to choose a suitable variable $X$. The remaining details are
left to the reader. 
\end{proof}
  Now we stochastically dominate the number of seeds revealed at a
  step conditioned on the subgraph revealed up to the previous step 
by a variable $Y$ with finite expectation which is independent of the
step number or $n$.
\begin{lem}\label{lem:mark_domination}
The number of vertices added to $R_{j-1}$ the $j$th step of the exploration process conditioned on $R_{j-1}$ is stochastically dominated by
a variable $Y$ with $\E Y < C$ where $C$ is a constant which do not depend upon $j$ or $n$.
\end{lem}
\begin{proof}

 Recall $r_i$ denotes the cardinality of the set $\{j:\lambda_j=i\}$. Now
note that because of the condition $(A)$, we can choose $\vartheta>1$ such
that	 $\sum_{i\ge 3} \vartheta ir_i <d_3n$ for some number $0 <d_3 <1$. Since $|R_k| <n^{1/9}$, the probability that the number of vertices added to $R_{j-1}$ in the $j$th step is $i$ for $i \ge 3$ is at most
$ir_i/(n-n^{1/9})<\vartheta ir_i/n$ for large enough $n$ using \cref{eq:analogue4}. Now define $Y$ as follows:
\begin{equation}
\P(Y=i) =
\begin{cases}
  \vartheta\frac{ir_i}{n} & \mbox{ if }i\ge 3\\
 1- \sum_{i\ge 3} \vartheta \frac{ir_i}{n} := p_2 & \mbox{ if } i=2\nonumber
\end{cases}
 \end{equation} 
Note further that
$$\E(Y)=2p_2+\vartheta \sum_{i\ge 3} i^2r_i/n < 2p_2 + \sum_{i=1}^N \lambda_i^2/n < 2+C_2 $$ from condition $(A)$. Thus clearly $Y$ satisfies the conditions of the Lemma. 
\end{proof}
Again, recall the definition of $v_{k+1}$ from \cref{lem:coupling_degree}. The following lemma is clear now.
\begin{lem}\label{lem:dom_lower}
Let $X,Y$ be distributed as in \cref{lem:mark_domination,lem:coupling_degree} and suppose they are mutually independent. Conditioned on $R_k$ such that $R_k$ do not have any bad pair, the number of vertices added to $R_k$ when we finish exploring $v_{k+1}$  is stochastically dominated
by a variable $Z$ where $Z$ is the sum of $X$ independent copies of
the variable $Y$. Consequently $\E Z < C$ where $C$ is a constant which do not depend upon $k$ or $n$. 
\end{lem}
\begin{proof}[Proof of \cref{thm:lower_marked_tree}]
We perform exploration process I. Let  $r_\delta$ be the maximum integer $r$ such that $\tau_r <
\delta$. Let $B^{\lambda}_r(V_x)$ denote the ball of radius $r$ around
the vertex $V_x$ in the underlying graph of $\Tg$. Recall that because
of \cref{prop:exploration_distance}, $B^{\lambda}_r(V_x)$ is obtained
by gluing together vertices with the same mark in $R_{\tau_r} = E_r$. Note that if $|B^{\lambda}_{\lf \ve \log n \rf}(V_x)| \le
n^{1/9}$ then the probability that $V_y$ lies in $B^{\lambda}_{\lf \ve
  \log n \rf}(V_x)$ is $O(n^{-8/9}\log n )$ because of condition
$(A)$ part (i). Hence it is enough to prove $\P_\lambda(r_\delta < \ve\log n)
\to 0$. Further, because of \cref{lem:death}, it is enough to prove
$\P_{\lambda}(r_\delta < \ve \log n \cap \mathcal B)\to 0$  where
$\mathcal B$ is the event that $S$ do not contain a bad pair. 

Consider a Galton-Watson tree with offspring distribution $Z$ as
specified in \cref{lem:dom_lower} and suppose $Z_r$ is the number of
offsprings in generation $r$ for $r \ge 1$. Then from \cref{lem:dom_lower}, we get
\begin{align}
 \P_{\lambda}(r_\delta < \ve \log n \cap \mathcal B) <
 \P_\lambda\left(\sum_{k=1}^{\lf \ve \log n \rf} Z_k > n^{1/9}\right)
 \to 0\label{eq:lower_marked_tree2}
\end{align}
if $\ve >0$ is small enough which follows from the fact that
$\E(Z_r) < C^r$ where $C$ is the constant in \cref{lem:dom_lower} and Markov's inequality. 
\end{proof}
Now we finish the proof of \cref{thm:inj_radius}.
\begin{proof}[Proof of \cref{thm:inj_radius}]
We shall use the notations used in the proof of \cref{thm:lower_marked_tree}.
Observe that if the ball of radius $r_\delta$ in the underlying graph
of $\Tg$ contains a circuit, then two connected components must coalesce
to form a single component at some step $k<\delta$. However this means
that there exists a bad pair.\ Thus on the event $\mathcal B$, the
underlying graph of $R_\delta$ do not contain a circuit. Hence on the
event $\mathcal B$, the ball of radius $\ve \log n$ contains a circuit in
the underlying graph of $\Tg$ implies $r_\delta < \ve \log n$. However
from \cref{eq:lower_marked_tree2}, we see that the probability of
$\{r_\delta<\ve \log n \cap \mathcal B\} \to 0$ for small enough $\ve>0$. The rest of the proof follows easily from \cref{lem:death,cor:conditionA,lem:simplify}. 
\end{proof}

Now we finish off by providing the proof of \cref{lem:tree_estimate}.
\begin{proof}[Proof of \cref{lem:tree_estimate}]
It is easy to see that  $$\P(d_0=j) = \frac{\Phi_{\sigma+j-1,n-e-j}}{\Phi_{\sigma,n-e}}$$ where $\Phi_{\sigma,n}$ is
given by \cref{eq:forest}.
A simple computation shows that 
\begin{align}
 \frac{\Phi_{\sigma+j-1,n-e-j}}{\Phi_{\sigma,n-e}} &=
\frac{\sigma+j-1}{\sigma}\frac{1}{2^{j}}\nonumber\\
\quad \quad \quad \quad \quad &\times\left(\frac{(n-e+\sigma)\prod_{i=1}^{j-1}(1-i/(n-e))}{(2(n-e)+\sigma-1) \prod_{i=2}^{j+1}(1+(\sigma-i)/2(n-e))}
\right) \label{eq:coupling_degree2}
\end{align}
Now we can assume
$(n-e+\sigma)/(2(n-e)+\sigma-1)\le 1$ (since $e \neq n$ by assumption). Also notice
\begin{equation}
 1-\frac{i}{n-e} < 1+\frac{\sigma-i}{2(n-e)}  \nonumber
\end{equation}
  for $i\ge 1$. Hence \cref{eq:coupling_degree2} yields
\begin{align}         
\frac{\Phi_{\sigma+j-1,n-e-j}}{\Phi_{\sigma,n-e}} &\le \frac{\sigma+j-1}{\sigma}\frac{1}{2^{j}}\left(\frac{(1-1/(n-e))}{ \prod_{i=j}^{j+1}(1+(\sigma-i)/2(n-e))}\right) \nonumber\\
&\le  \frac{\sigma+j-1}{\sigma}\frac{4	}{2^{j}} \le \frac{4j}{2^j}\label{eq:tree_estimate1}
\end{align}
 which follows because $\prod_{i=j}^{j+1}(1+(\sigma-i)/2(n-e)) \ge
 1/4$ since $n-e \ge j$ and for the second inequality of
 \eqref{eq:tree_estimate1}, we
use the trivial bound $(\sigma+j-1)/\sigma \le j$. 

Further note that $\P(d_0 = k,d_1 = j-k )$ for any $0\le k\le j$ is
given by $\Phi_{\sigma+j-2,n-e-j}/\Phi_{\sigma,n-e}$. Hence summing over
$k$, $$\P(d_0
+d_1 = j)=(j+1)
\frac{\Phi_{\sigma+j-2,n-e-j}}{\Phi_{\sigma,n-e}}.$$ Now keeping $n$ fixed,
$\Phi_{\sigma,n}$ is an increasing function of $\sigma$, hence using
the bound obtained in \eqref{eq:tree_estimate1}, the proof is complete.
\end{proof}
  

\section{Upper Bound}\label{sec:upper_diameter}
Throughout this section, we again fix a $\lambda$ satisfying condition
$(A)$ as described in \cref{sec:odd_permutation}.
Recall $d_\lambda(.,.)$ denotes the graph distance metric in the underlying
graph of $\Tg$. In this section we prove the following Theorem.
\begin{thm}\label{thm:upper_marked_tree}
Fix a $\lambda$ satisfying condition $(A)$. Suppose $V_1$ and $V_2$ be
vertices corresponding to the marks $1$ and $2$ in $\Tg$.
Then there exists a constant $C>0$ such that 
\[
 \P_{\lambda}(d_{\lambda}(V_{1},V_{2}) >C\log n)  =O(n^{-3})
\]
\end{thm}
Note that the distribution of $\Tg$ is
invariant under permutation of the marks. Hence the choice of marks $1$
and $2$ in 
\cref{thm:upper_marked_tree} plays the same role as an arbitrary pair
of marks.

\begin{proof}[Proof of \cref{thm:diameter} part (ii)]
Proof follows from \cref{thm:upper_marked_tree,lem:simplify,cor:conditionA}.
\end{proof}

To prove \cref{thm:upper_marked_tree}, we
plan to use an exploration process similar to that in
\cref{sec:lower_diameter} albeit with certain modification to overcome technical
hurdles. We
start the exploration process from a vertex $v_1$ with mark $1$ and continue to explore for roughly $n^{3/4}$ steps. Then we
start from the vertex $v_2$ with mark $2$ and explore for another $n^{3/4}$
steps. Since the sets of vertices revealed are approximately
uniformly and
randomly selected from the set of vertices of the tree, the distance
between these sets of vertices should be small with high probability,
because of the same reasoning as the birthday paradox problem. Then we show that
the distance in the underlying graph of $\Tg$ from the set of vertices
revealed and $1$ or $2$ is roughly $\log n$ to complete the proof. To
this end, we shall find a supercritical Galton-Watson tree whose
offspring distribution will
be dominated by the vertices revealed in every step of the process.

However, if we proceed as the exploration process described in \cref{sec:lower_diameter}, since an
unexplored vertex has a reasonable chance of being a leaf, the
corresponding Galton-Watson tree will also have a reasonable chance of
dying out. However, we need the dominated tree to survive for a long
time with high probability. To overcome this difficulty, we shall invoke the
following trick. Condition on the tree $U_{0,n}$ to have diameter $\gg \log^2n$. Consider the vertex $v_*$
which is farthest from $\{v_1,v_2\}$. For each vertex we explore, we reveal
its unique neighbour which lie on the path joining the vertex and
$v_*$ instead of revealing all the neighbours which do not
lie in the set of revealed vertices. Note that the
revealed vertices by the exploration process now will mostly be
disjoint paths increasing towards $v_*$ and we shall always
have at least one child if the paths do not
intersect. However the
chance of paths intersecting is small. Since expected size of
$mark(v)$ for any non-revealed vertex $v$ is larger than $1$
throughout the process, we have exponential growth accounting for the
logarithmic distance. The
rest of the Section is devoted to rigorously prove the above described
heuristic.

We shall now give a brief description of the exploration process 
we shall use in this section which is a modified version
of exploration process described in \cref{sec:lower_diameter}. Hence, we shall not write down details of
the process again to avoid repetition, and concentrate on the
differences with exploration process I as described in
\cref{sec:lower_diameter}.

\paragraph{Conditioning on the tree:} For the proof of
\cref{thm:upper_marked_tree}, we only need randomness of the marking
function $M$ and not that of the tree $U_{0,n}$. Hence, throughout this section, we shall condition on a plane tree $U_{0,n} = t$ where $t \in \mathcal U_{0,n}$ such that
\begin{mylist}
 \item $diam(t) > \sqrt{n}/\log n$.
\item $M_{\lf \log^3n \rf} \le \log^8 n$.
\end{mylist}
where recall that $M_{r}$ is as defined in \cref{lem:max_ring}:
maximum over all vertices $v$ in $U_{0,n}$ of the volume of the ball of
radius $r$ around $v$. Let
us call this condition, \textbf{condition $(B)$}. Although apparently
it should
only help if the diameter of $t$ is small, the present proof fails to
work if the diameter is too small and requires a different argument
which we do not need. Note that by
\cref{lem:max_ring,lem:diameter_tree}, the probability that $U_{0,n}$
satisfies condition $(B)$ is at least $1-\exp(-c\log^2n)$ for some
constant $c>0$. Hence it is enough to prove
\cref{thm:upper_marked_tree} for the conditional measure which we
shall also call $\P_{\lambda}$ by an abuse of notation.

We start with a marked tree $(t,m)$ where $t$ satisfies condition $(B)$. As planned, the exploration process will proceed in two stages, in the first
stage, we start exploring from a vertex with mark $1$ and in the
second stage from a vertex with mark $2$.
\begin{description}
\item[\underline{Exploration process II, stage 1:}]\hspace*{1mm}
There
will be three states of vertices \textbf{active, neutral} or \textbf{dead}. We shall
again define a nested sequence of subgraphs $R_0 \subseteq R_1 \subseteq R_2 \subseteq \ldots$
which will denote the subgraph revealed. Alongside
$\{R_k\}_{k=0,1,\ldots}$, we will define another nested sequence
$Q_0 \subseteq Q_1 \subseteq Q_2 \subseteq \ldots$ which will denote \textbf{dead vertices
revealed}. 

We shall similarly
define the sequences $\{N_r\}$, $\{ E_r\}$ and $\{\tau_r\}$ as in
exploration process I. We call $v$ to be a $v_*$-ancestor of another vertex $v'$
if $v$ lies on the unique path joining $v'$ and $v_*$. The
$v_*$-ancestor which is also the neighbour of $v$ is called the
$v_*$-parent of $v$.

\item[Starting Rule:] We start from a vertex $v_1$ with mark $1$ and
  $v_2$ with mark $2$ (if there
  are more than one, select arbitrarily.) Let $v_*$ be a
  vertex farthest from $\{v_1,v_2\}$ in $t$ (break ties arbitrarily.)
  Note that because of the lower bound on the diameter via condition
  $(B)$, $d^t(v_1,v_*)$ and $d^t(v_2,v_*)$ are at least
  $\sqrt{n}(3\log n)^{-1}$. Declare $v_1$ to be active and let $R_0 =
  \{v_1\}$. Declare all the vertices in $mark(v_1)\setminus v_1$ to be
  dead and let $Q_0 =
  mark(v_1)\setminus v_1$. Set $\tau_0 = 0$ and $E_0 = R_0$.

\item[Growth rule:] Suppose we have defined $E_0\subseteq \ldots \subseteq E_r$, $\tau_0 \le \ldots \le \tau_r$ and also $R_0 \subseteq \ldots \subseteq R_{\tau_r}$ such that $R_{\tau_r} = E_r$ and $N_r : = E_{r} \setminus E_{r-1}$ is the set of active vertices in $E_r$. Now we explore vertices in $N_r$ in some predetermined order and suppose we have determined $R_k$ for some $k \ge \tau_r$. We now move on to the next vertex in $N_r$. If there is no such vertex, declare $k = \tau_{r+1}$ and $E_{r+1} = R_{\tau_{r+1}}$.

Otherwise suppose $v$ is the vertex to be
  explored in the $k+1$th step. Let $v_-$ denote the
  $v_*$-ancestor which is not dead and is nearest to $v$ in the tree
  $t$. If $v_-$ is already in $R_k$ then we terminate the
  process. 
\item[Death rule:] Otherwise, declare $v_-$ to be active, $v$ to be
  neutral and let $\Lambda = mark(v_-) \setminus v_-$. If any
  vertex $u \in \Lambda$ is within distance $\log^3n$ from $R_k \cup
  Q_k \cup v_*$ or another $u' \in \Lambda$, we say {\em death rule is
  satisfied} (see \cref{fig:deathrule1}.) If death rule is satisfied declare all the vertices in
$\Lambda$ to be dead and set $Q_{k+1} = Q_k \cup \Lambda$, $R_{k+1} =
R_k \cup v_-$. Otherwise declare all the vertices in $\Lambda$ to be
active, set $R_{k+1} = R_k \cup mark(v_-)$ and $Q_{k+1} = Q_k$. 

\begin{figure}[t]
\hspace{14mm}\centering{\includegraphics[width=1 \textwidth]{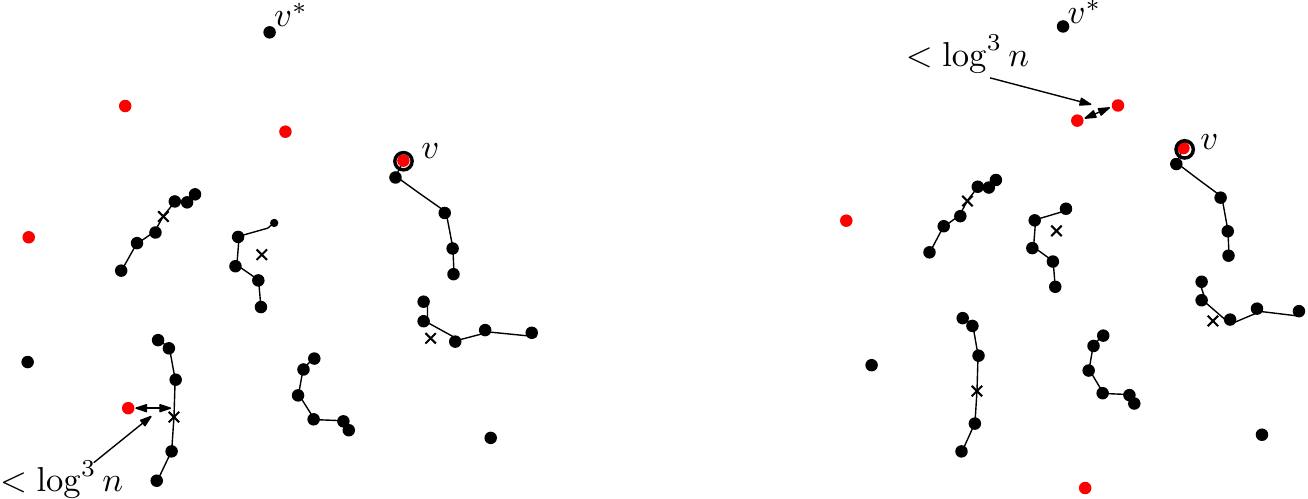} \hfill }
\caption{Illustration of the death rule in exploration process II. A
  snapshot of the revealed vertices when we are exploring the circled vertex
  $v$ is given. The black vertices and edges correspond to neutral and
  active vertices, while the crosses correspond to dead
  vertices. We are exploring $v$ and $mark(v)$ is denoted by the red vertices. On
  the left: a red vertex comes within distance $\log^3n$ of one the revealed
  vertices, hence death rule is satisfied. On the right: two of the
  revealed vertices are within distance $\log^3n$. Hence death rule is
satisfied.}\label{fig:deathrule1}
\end{figure}
\item[Threshold rule:] We stop if the number of steps exceed $n^{3/4}$
  or $r$ exceeds $\log^2n$. Let $\delta$ denote the step when we stop stage $1$ of
  the exploration procedure.

\item[\underline{Exploration process II, stage 2:}]\hspace*{1mm}  
Similarly as in stage $1$, we start with $R'_0 = {v_2}$ being active
and \mbox{$Q'_0 = mark(v_2) \setminus v_2$} being dead.
We proceed exactly as in stage $1$, except for the following
change: if $v_2$ is a neighbour of $R_\delta \cup Q_\delta$, or while exploring $v$, if any of the vertices in $mark(v_-)$ is a neighbour
  of $R_\delta \cup Q_\delta$, we say a \textbf{collision} has
  occurred and terminate the procedure.

\end{description}

We shall see later (see \cref{lem:kappa,cor:kappa}) that with high probability, we perform the exploration
for $n^{3/4}$ steps and the number of rounds is approximately $\log n$
in stage $1$. Also in stage $2$, collision occurs with high
probability and the number of rounds is at most
$\log n$ with high probability.

In what follows, we shall denote by $X'$ in stage $2$ the set or variable corresponding to that
denoted by $X$ in stage $1$ (for example, $R_k',Q_k'$ will denote the
set of revealed subgraphs and dead vertices respectively up to stage
$k$ in stage $2$ etc.)

Now we shall define a new tree $T_C$ which is defined on a subset of
vertices of $t$. The tree $T_C$ will capture the growth process
associated with the exploration process. We start with the tree $t$ and remove all its edges so we are left with only its vertices. The root vertex of $T_C$ is $v_1$. For every step of exploring $v$, we add an
edge beween $v$ and every vertex of $mark(v_-)$ (where $v_-$ is defined as in growth
rule) in $T_C$ if death rule is not
satisfied. Otherwise we add an edge between $v$ and $v_-$ in
$T_C$. The vertices we connect by an edge to $v$ while exploring $v$ is called
the offsprings of $v$ in $T_C$ similar in spirit to a Galton-Watson tree. It is
clear that $T_C$ is a tree (since we terminate the procedure if $v_-
\in R_k$). Let $Z_r$ denotes the number of vertices
at distance $r$ from $v_1$ in $T_C$. Clearly, if we glue together
vertices with the same mark which are at a
distance at most $r$ in $T_C$, we obtain a subgraph of the ball of
radius $r$ in the underlying graph of $\Tg$. We similarly define another tree corresponding to
stage $2$ of the process which we call $T_C'$ which starts from the
root vertex $v_2$.

\begin{lem}\label{lem:total_volume}
The volume of $R_{\delta} \cup Q_{\delta}$ is at most $C_0n^{3/4}\log n$. Also
the volume of $R_{\delta'} \cup Q_{\delta'}$ is at most $C_0n^{3/4}\log n$
where $C_0$ is as in part (i) of condition $(A)$.
\end{lem}
\begin{proof}
In every step, at most $C_0\log n$ vertices are revealed by condition $(A)$.
\end{proof}

Define $v_1$ and $v_2$ to be the \textbf{seeds} revealed in the first step.
While exploring vertex $v$, we call the vertices in $mark(v_-)
\setminus v_-$ to be the seeds revealed at that step if the death rule is not satisfied. Note
that because of the prescription of the death rule, seeds are
necessarily isolated vertices (not a neighbour of any other revealed
neutral or dead vertex up to that step.)

\begin{defn}
 A \textbf{worm} corresponding to a seed $s$ denotes a sequence of vertices $\{w_0,w_1,w_2,\ldots,w_d\}$ such that $w_0$ is $s$ and $w_{i+1}$ is the vertex ${w_i}_-$ for $i \ge0$.
\end{defn}
Note that in the above definition, ${w_i}_-$ depends on the vertices revealed up to the time we explore $w_i$ in exploration process II.
 \begin{figure}[t]
\hspace{14mm}\centering{\includegraphics[width=.5 \textwidth]{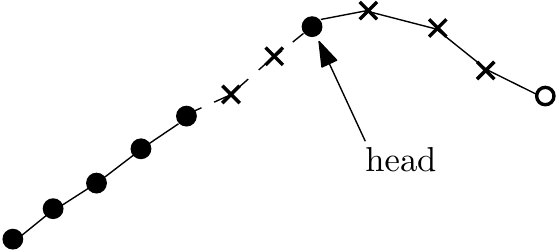} \hfill }
\caption{An illustration of a worm at a certain step in the
  exploration process. The black vertices 
  denote the vertices of the worm, the crosses are the
  dead vertices and the head of the worm is as shown. The circle is
  the $v_*$-ancestor of the head which is not dead. This worm has
  faced death $5$ times so far.
}\label{fig:disaster}
\end{figure}
Note that in a worm, $w_{i+1}$ is a neighbour of $w_i$ if the
$v_*$-parent of $w_i$ is
not a dead vertex. If it is a dead vertex we move on to the next nearest ancestor
of $w_i$ which is not dead. Note that the ancestors of $w_i$ which lie on the path joining
$w_{i+1}$ and $w_i$ are necessarily dead. If there are $p$ 
dead vertices on the path between $w_0$ and the nearest $v_*$-ancestor
of $w_d$ which is not dead, we say that the worm
has \textbf{faced death $p$ times} so far (see \cref{fig:disaster}.) Call $w_d$ the \textbf{head} of the
worm. The \textbf{length} of the worm is the distance in $U_{0,n}$ between $w_d$ and $w_0$.

We want the worms to remain disjoint so that conditioned up to the
previous step, the number
of children of a vertex in the tree $T_C$ or $T_C'$ remain independent of the
conditioning. Now for any worm, if $ w_{i} \in N_r$ (resp. $w_i \in  N_r'$), then it
is easy to see that $w_{i+1} \in N_{r+1}$ (resp. $ w_{i+1} \in N_{r+1}'$) because of the way the
exploration process evolves. Hence if none of the worms revealed
during the exploration process face a dead vertex, then the length of
each worm is at most $\log^2n$ from threshold rule. Since every
seed is at a distance at least $\log^3n$ from any other seed via the death
rule, the worms will remain disjoint from each other if death does not
occur.

Unfortunately, many worms will face a dead vertex with reasonable
chance. But fortunately, none of them will face many dead vertices
with high probability. We say that a \textbf{disaster} has occurred at step
$k$ if after performing step $k$, there is a worm which has faced death at least $16$ times. The following proposition is immediate from the threshold rule for exploration process II and the discussion above.

\begin{prop}\label{prop:face_dead}
If disaster does not occur, then the length of each worm is at most
$\log^2n +16$ and hence no two worms intersect during the exploration
process for large enough $n$. 
\end{prop}

We will now provide a series of Lemmas using which we will prove \cref{thm:upper_marked_tree}. The proofs of \cref{lem:face_dead,lem:offspring} are postponed to \cref{sec:rem_proofs} for clarity.

We start with a Lemma that shows that disaster does not happen with
high probability and consequently the length of
each worm is at most $\log^2n+16$ with high probability. 

\begin{lem}\label{lem:face_dead}
With $\P_{\lambda}$-probability at least $1-cn^{-3}$ disaster does not occur where $c>0$ is some constant. 
\end{lem}

\begin{lem}\label{lem:distance_control}
 Suppose disaster does not occur and $r \ge 0$.  Then in the tree $T_C$, for 
 $v \in Z_r$, $d^{T_C}(v,v_1) < 16r$. Also in $T_C'$,
 for $v \in Z_r'$, $d^{T_C'}(v,v_2) < 16r$. 
\end{lem}
\begin{proof}
We prove only for stage $1$ as for stage $2$ the proof is
similar. The Lemma is trivially true for $r =0$. Suppose now the Lemma is true
for $r'=r$. Now for any vertex in $Z_r$ and its offspring, the
vertices corresponding to their marks in the underlying graph of $\Tg$
must lie at a distance at most $16$ because otherwise disaster would
occur. Hence the distance of every vertex in $Z_{r+1}$ from $v_1$ is at most $16r+16 = 16(r+1)$. We use induction to complete the proof.
\end{proof}

Let $\mathcal E$ (resp. $\mathcal E'$) be the event that disaster has
not occurred up to step $\delta$ (resp. $\delta'$). Let $\mathcal F_k$ denote the sigma field generated by
\mbox{$R_0, \ldots, R_k,Q_0,\ldots,Q_k$}. Let $\mathcal
F_k'$ denote the sigma field $\sigma(R_0', \ldots,
R_k',Q_0',\ldots,Q_k') \vee \mathcal F_\delta$. Let $v_{k+1}$ (resp. $v_{k+1}'$) be the vertex we explore in the $k+1$th stage in the exploration process stage $1$ (resp. stage $2$).

\begin{lem}\label{lem:offspring}
Conditioned on $\mathcal F_k$ (resp. $\mathcal F_k'$) such that
$k<\delta$ (resp. $\delta'$) and disaster has not
occurred up to step $k$, the
$\P_\lambda$-probability that $v_{k+1}$ in $T_C$ (resp. $T_C'$) has 
\begin{mylist}
 \item no offsprings is $0$. 
\item at least $3$ offsprings is at least $k_0$ for some constant $k_0 > 0$.
\end{mylist}
\end{lem}

We will now construct a supercritical Galton-Watson tree which is
stochastically dominated by both $T_C$ and $T_C'$.
Consider a Galton-Watson tree $GW$ with offspring distribution $\xi$ where 
\begin{itemize}
 
\item $\P(\xi = 1) = (1-k_0/2) $
\item $\P(\xi = 2) = k_0/2$
\end{itemize}
and $k_0$ is the constant obtained in part (ii) of \cref{lem:offspring}.
Let $Z^{GW}_r$ be the number of offsprings in the $r$-th generation of $GW$. 
\cref{lem:offspring} and the definition of $GW$ clearly shows
that $Z_r$ stochastically dominates $Z^{GW}_r$ for all $r \ge 1$ if
disaster does not occur up to step $\tau_r$. Let $r_{\delta} =
\max\{r: \tau_r <\delta \}$ and similarly define
$r_{\delta'} = \max\{r: \tau_r <\delta' \}$. Thus, we have

\begin{lem}\label{lem:dominance}
For any integer $j \le r_{\delta}$ (resp. $j \le r_{\delta'}$), $Z_j$
stochastically dominates $Z^{GW}_j$ on the event $\mathcal E$
(resp. $\mathcal E'$).
\end{lem}
It is clear that the mean offspring distribution of $GW$ is strictly
greater than $1$ and hence $GW$ is a supercritical Galton-Watson
tree. Also $GW$ is infinite with probability $1$.

Now we are ready to show that the depth of $T_C$ (resp. $T_C'$) when we
run the exploration upto time $\delta$ (resp. $\delta'$) is of logarithmic order with high
probability. 
\begin{lem}\label{lem:kappa}
 There exists a $C>0$ such that 
\begin{mylist}
\item $ \P_\lambda((r_{\delta} > C\log n) \cap \mathcal E) =O(n^{-3})$
\item $ \P_\lambda((r_{\delta'} > C\log n) \cap \mathcal E') =O(n^{-3})$
\end{mylist}

\end{lem}
\begin{proof}
 We shall prove only (i) as proof of (ii) is similar. 
 Because of \cref{lem:dominance,lem:total_volume} we have for a large
 enough choice of
 $C>0$, 
\begin{align}
 \P_\lambda((r_{\delta}>C\log n) \cap
 \mathcal E) & < \P_\lambda\left(\sum_{i=1}^{\lf 
C\log n
\rf}Z^{GW}_i \le C_0 n^{3/4} \log n\right) \nonumber\\
&<\P_\lambda\left(Z^{GW}_{\lf C\log n \rf} \le C_0 n^{3/4} \log n\right) \nonumber\\
&<n^{-3}\label{eq:kappa1}
\end{align}
where \eqref{eq:kappa1} follows by applying \cref{lem:lowerdev},
choosing $C>0$ large enough and observing the fact that
$\P_\lambda (Z_r=0) = 0$ for any $r$ from definition of $GW$.
\end{proof}
Recall that in stage $2$, we stop the process if we have revealed a
vertex which is a neighbour of $R_\delta \cup Q_\delta$, and we say a\
\textbf{collision} has occurred. Let us denote the event that collision does
not occur up to step $k$ by $\mathcal C_k$.
Since $\delta \le \lf n^{3/4} \rf$ implies either disaster has
occurred in stage $1$
or $r_{\delta}>\log^2 n$ and $\delta' \le \lf n^{3/4} \rf$ implies
either disaster has occurred in stage $2$
or $r_{\delta}'>\log^2 n$ or a collision has occurred  we have the immediate corollary
\begin{corollary}\label{cor:kappa}
 On the event $\mathcal E$, the $\P_{\lambda}$-probability that
 $\delta \le \lf n^{3/4} \rf$ is $O(n^{-3})$. On the event $\mathcal E' \cap \mathcal C_{\delta'}$, the $\P_{\lambda}$-probability that
 $\delta' \le \lf n^{3/4} \rf$ is $O(n^{-3})$.
\end{corollary}

   \medskip
Now we are ready to prove our estimate on the typical
distances.
We show next, that a collision
will occur with high probability.
\begin{lem}\label{lem:collision}
Probability that a collision occurs before step $\delta'$ is at least
$1-cn^{-3}$ for some constant $c>0$.
\end{lem}
\begin{proof}
 Let $\mathcal H$ be the event that disaster does not occur up to step
 $\delta$, $\delta = \lf n^{3/4} \rf + 1$. Let $A(R_{\delta})$ be the
 set of $v_*$-
 parents of the heads of the worms in
 $R_{\delta}$. Since at each step at least
 one vertex is revealed, $\delta = \lf n^{3/4}
 \rf + 1$ implies the number of vertices revealed is at least
 $n^{3/4}$. If disaster does not occur, then from
 \cref{lem:face_dead,prop:face_dead}, the worms are disjoint and each
 worm has length at most $\log^2n + 16$. Hence the number of vertices in $A(R_\delta)$ is
at least $n^{3/4}/(\log^2 n+16)$.
Also, the number of vertices in $A(R_\delta)$ is at most
$C_0n^{3/4} \log n$ from \cref{lem:total_volume}. For any $k<\delta'$,
conditioned on the event $\mathcal C_k$ that no collision has occured
up to step $k$, the probability that collision occurs in step $k+1$
when we are exploring a vertex $v$
is at least (using Bonferroni's inequality),
\begin{align}
  &  \sum_{w\in A(R_\delta)}\P_{\lambda}(m(v) = m(w)|\mathcal C_k,\mathcal H) 
-\sum_{w,z \in A(R_\delta)}\P_\lambda(m(v)=m(w) = m(z) | 
\mathcal C_k,\mathcal H)\nonumber\\
        & > \frac{c}{n^{1/4}\log^2 n} -
\frac{c\log^2n}{\sqrt{n}}\label{eq:collision1}\\
        & > \frac{c}{n^{1/4}\log^2 n}\label{eq:collision2}
\end{align}
for some constant $c>0$. The first term of \eqref{eq:collision1} follows from the lower bound of \eqref{eq:analogue3}. The second term of
\eqref{eq:collision1} follows from \eqref{eq:analogue2} and noting that the number of terms in the sum is
$O(n^{3/2}\log^2n)$. Since the bound on the
probability displayed in \eqref{eq:collision2} is independent of the
conditioning, 
\begin{align}
 &\P_\lambda(\mathcal C_{\delta'}\cap \delta'=\lf n^{3/4} \rf+1 | \mathcal H) +
 \P_\lambda(\mathcal C_{\delta'} \cap \delta'\le \lf n^{3/4} \rf | \mathcal H) \nonumber\\
<&\left(1 - \frac{c}{n^{1/4}\log^2 n}\right)^{n^{3/4}}  +
\P_\lambda(\mathcal C_{\delta'}
\cap \mathcal E' \cap \delta'\le \lf n^{3/4} \rf | \mathcal H)+
\P_\lambda((\mathcal E')^c|\mathcal H)\nonumber\\
<&\exp(-c\sqrt{n}/\log^2 n) + O(n^{-3}) \label{eq:collision3}\\
 =& O(n^{-3})\nonumber
\end{align}
 where the bound on the second term in \cref{eq:collision3} follows from \cref{cor:kappa,lem:face_dead}. The Lemma now follows because the probability of the complement of
$\mathcal H$ is $O(n^{-3})$ again from \cref{cor:kappa,lem:face_dead}.
\end{proof}
\begin{proof}[Proof of \cref{thm:upper_marked_tree}]
Suppose we have performed exploration process I stage 1 and 2. Let
$\mathcal G$ be the event that $r_{\delta} \le C \log n, r_{\delta'}
\le C \log n$, disaster does not occur before step $\delta$ or
$\delta'$ and a collision occurs. On the event $\mathcal G$ the distance between $V_1$ and $V_2$ in the underlying graph of $\Tg$ is at most $32C \log n+1$ by \cref{lem:distance_control}. But by \cref{lem:face_dead,lem:collision,lem:kappa}, the complement of the event $\mathcal G$ has probability $O( n^{-3})$. 
\end{proof}
   \subsection{Remaining proofs}\label{sec:rem_proofs}
The proofs of both the Lemmas in this subsection are for stage $1$ of the
exploration process as the proof for stage $2$ is the same.

 \begin{proof}[Proof of \cref{lem:face_dead}]
Let $s$ be a seed revealed in the $k$th step of exploration process II
. Suppose $P$ denotes the set of vertices
at a distance at most $\log^2n+16$ from $s$ along the unique path
joining $s$ and $v_*$. Note that none of the vertices in $P$ are revealed yet
because of the death rule. If the worm corresponding to $s$ faces more
than $16$ dead vertices, then more than $16$ dead vertices must be
revealed in $P$ during the exploration from step $k$ to $\delta$.
Conditioned up to the previous step, the probability that one of the
revealed vertices lie in $P$ in a step is $O(n^{-1}(\log^2n+16))$ from
\eqref{eq:analogue3} and union bound. Since
this bound is independent of the conditioning, the probability
that this event happens at least $16$ times during the process is $O(n^{-16}\log^{32}n \cdot n^{12}) = O(n^{-4}
\log^{32}n)$ where the factor $n^{12}$ has the justification that $\dbinom{\lfloor n^{3/4}\rfloor}{16}=O(n^{12})$ is the number of combination of steps by which this event can happen $16$ times. Observe that more than one vertex may be revealed in
$P$ in a
step, but the probability of that event is even smaller.\ Thus taking union over all seeds, we see
that the probability of disaster occurring is
$O(n^{-13/4}\log^{33}n) = O(n^{-3})$ using \cref{lem:total_volume} and union bound.  
\end{proof}
  \begin{proof}[Proof of \cref{lem:offspring}]
 It is clear that on the event of no
 disaster, every explored vertex has at least the offspring corresponding
 to its closest non-dead $v_*$-ancestor in the tree $T_C$. This is because on the event of no disaster,
 no two worms intersect. For stage $2$, the closest non-dead $v_*$-ancestor cannot
 belong to $R_\delta \cup Q_\delta$ because otherwise, the process
 would have stopped. Hence (i) is trivial.

Now for (ii), first recall that condition $(A)$ ensures that the number of indices $i$ such that $\lambda_i \ge 3$ is at
least $(1-d_2)n$. Now since the number of
vertices revealed upto any step $k < \delta$ is $O(n^{3/4}\log n )$, the number of vertices left with mark $i$ such that $\lambda_i \ge 3$ is at least $(1-d_2)n - O(n^{3/4}\log n )>
(1-d_2)n/2$ for large enough $n$. Note that the number of offsprings of
$v$ in $T_C$ is at least $3$ if the number of vertices with the same
mark as $v_-$ is at least $3$ and death does not occur. Hence if we
can show that the probability of death rule being satisfied in a step
is $o(1)$, we are done.

To satisfy the death rule in step $k+1$, a vertex in $mark(v_-) \setminus v_-$ must
be within distance $\log^3n$ in the tree $U_{0,n}$ to another vertex in $mark(v_-) \setminus v_-$ or $R_k \cup
Q_k \cup v^*$. Now using of part (ii) of condition $(B)$ and \cref{lem:total_volume}, the number
of vertices within $\log^3n$ of $R_\delta \cup Q_\delta \cup v^*$ is
$O(n^{3/4}\log^9n)$. Hence the probability that the death rule is
satisfied is  $O(n^{-1/4}\log^9n) = o(1)$ by union bound. This
completes the proof.
 \end{proof}

\bibliographystyle{abbrv}
\bibliography{higenus}

\begin{appendix}\label{sec:appendix}
\section{Proof of \cref{cor:conditionA}}\label{sec:proofs_perm}
We shall prove \cref{cor:conditionA} in this section. We do the computation following the method of random
allocation similar in lines of \cite{Kaz}. For this,
we need to introduce i.i.d. random variables $\{\xi_1, \xi_2, \ldots\}$ such that for
some parameter $\beta \in (0,1)$
\begin{equation}\label{eq:moment}
P(\xi_1 = 2i+1)
\begin{cases}
= \frac{\beta^{2i+1}}{B(\beta)(2i+1)} & \mbox{if } i \in \N\cup\{0\}\\
= 0 & \mbox{otherwise}
\end{cases}
\end{equation} 
where $B(\beta) = 1/2 \log((1+\beta)/(1-\beta))$. Recall that
$\mathcal P$ is the set of all $N$-tuples of odd positive integers
which sum up to $n+1$. Observe that for any 
$z=(z_1,\ldots,z_N) \in \mathcal P$,
$$
P(\lambda=z) = P(\xi_1 =z_1,  \ldots ,\xi_N=z_N |
\xi_1+\xi_2+\ldots \xi_N = n+1 )
$$ 

throughout this Section, we shall assume the following:

\begin{itemize}
 \item $\{n,N\}\to \{\infty,\infty\}$ and $n/N \to \alpha$ for some constant
$\alpha>1$.
\item For every $n$, the parameter $\beta=\beta(n)$ is chosen such that $E(\xi_1) = m =
(n+1)/N$
\end{itemize}

It is easy to check using \eqref{eq:moment} that there is a unique
choice of such $\beta$ and $\beta$ converges to some finite
number $\beta_\alpha$ such that
$0<\beta_\alpha<1$ as $(n+1)/N$
converges to $\alpha$.
Let $\zeta_{N,j} = \xi_1^j + \ldots \xi_N^j$ where $j \ge 1$ is an integer. 
It is also easy to see that for any integer
$j\ge 1$, $\E\xi_1^j  = m_j(n)$ for
some function $m_j$ which also converge to some number $m_{j\alpha}$
as $n \to \infty$.  Let $\sigma_j^2 =
Var(\xi_1^j)$. To simplify notation, we shall denote $\zeta_{N,1}$ by $\zeta_N$
and $\sigma_1$ by $\sigma$.

We will first prove a central limit theorem for $\zeta_{N}$.
\begin{lem}\label{lem:CLT}
We have,
\begin{equation}
 \frac{\zeta_{N} - Nm}{\sigma\sqrt{N}} \to N(0,1)\label{eq:CLT}
\end{equation}
in distribution as $n \to \infty$.
\end{lem}
\begin{proof}
 Easily follows by checking the Lyapunov condition for triangular arrays of random variables
(see \cite{Durbook}). 
\end{proof}
We now prove a local version of the CLT asserted by \cref{lem:CLT}.
\begin{lem}\label{lem: denominator}
 We have
$$
P(\zeta_N = n+1) \sim \frac{2}{\sqrt{2\pi N }\sigma}
$$

\end{lem}
\begin{proof}

Let $\overline{\xi_i} = (\xi_i-1)/2$.
Apply Theorem $1.2$
of \cite{localCLT} for the modified arrays $\{\overline{\xi_1}, \ldots
\overline{\xi_N}\}_{n\ge 1}$ and use \cref{lem:CLT}. The details
are left for the readers to check.
\end{proof}

\begin{lem}\label{lem:moment_control}
Fix $j \ge 1$. There exists constants $C_1>1$ and $C_2>1$ (both depending only on $\alpha$ and
$j$) such that
\begin{equation}
P\left(C_1 n <\sum_{i=1}^N \lambda_i^j <C_2n\right) > 1-\frac{c}{n^{7/2}}
\end{equation}
for some $c>0$ which again depends only on $\alpha$ and $j$ for large enough
$n$.
\end{lem}
\begin{proof}
It is easy to see that $m_j \to m_{j\alpha}$ as $n \to \infty$.
Notice that 
\begin{align}
 P\left(\sum_{i=1}^N \lambda_i^j >C_2n, \sum_{i=1}^N \lambda_i^j
   <C_1n\right)& = P\left(\zeta_{N,j} >C_2n, \zeta_{N,j} <C_1n|\zeta_N =
n\right)\nonumber\\
& < \frac{P\left(\zeta_{N,j} >C_2n,  \zeta_{N,j} <C_1n\right)}{P(\zeta_N =
n)}\label{eq:moment1}
\end{align}
Choose $C_2>m_{j\alpha}$ and $C_1<m_{j\alpha}$. Then for some $c>0$,
for large enough $n$,
\begin{equation}
P\left(\zeta_{N,j} >C_2n, \zeta_{N,j} <C_1n\right) < P\left(|\zeta_{N,j} - m_jN| >cn\right)<\E\left(\zeta_{N,j} - m_jN\right)^8(cn)^{-8}\label{eq:moment2}
\end{equation}
It is easy to see that $\E\left(\zeta_{N,j} - m_jN\right)^8 = O(n^4)$
since the terms involving $\E(\xi_i^j-m_j)$ vanishes and all the
finite moments of $\xi_1$ are bounded. Now plugging in this estimate into
\cref{eq:moment2}, we get
\begin{equation}
P\left(\zeta_{N,j} >C_2n, \zeta_{N,j} <C_1n\right) =O(n^{-4})\label{eq:moment3}
\end{equation}  
Now plugging in the estimate of \cref{eq:moment3} into
\cref{eq:moment1}
and observing that $P(\zeta_N =
n) \asymp N^{-1/2}$ via \cref{lem: denominator}, the result follows.
\end{proof}
\medskip

\begin{lem}\label{lem:largest_cycle}
There exists a constant $C_0 > 0$ such that 
\[
P(\lambda_{\max} > C_0\log n) =O(n^{-3})
\]
\end{lem}
\begin{proof}
Let $\xi_{\max}$ be the maximum among $\xi_1,\ldots, \xi_N$. Note that 
\begin{equation}
\P(\xi_{\max} >C_0\log n)< N\P(\xi_1 >C_0\log n) =O (N\beta^{C_0\log n})=O(n^{-7/2})\label{eq:largest_cycle}
\end{equation}
if $C_0>0$ is chosen large enough. Now the
Lemma follows from the estimate of \cref{lem: denominator}. The
details are left to the reader.
\end{proof}

\begin{lem}\label{lem:fixed_points}
There exists
constants $0<d_1<1$ and $0<d_2<1$ which depends only on $\alpha$ such that
$\P(d_1n<|i:\lambda_i =1|<d_2n)<e^{-cn}$ for some constant $c>0$ for large enough $n$.
\end{lem}
\begin{proof}
The probability that $|i:\lambda_i =1| < d_2n$ for large enough $n$ for some $0<d_2<1$
follows directly from the fact that $|i:\lambda_i =1|\le N$. For the upper bound, 
\begin{align}
\P(|i:\lambda_i =1|>d_1n)=  \P\left(\sum_{i=1}^N1_{\lambda_i=1} >d_1n\right) <
\frac{\P\left(\sum_{i=1}^N1_{\xi_i=1} >d_1n\right)}{\P(\zeta_N =
n)}\label{eq:fixed_points1}
\end{align}
Now $\P(\xi_i = 1) \to \beta_\alpha/B(\beta_\alpha)$ as $n \to \infty$. The Lemma now
follows by choosing $d_1$ small enough, applying \cref{lem: denominator} to the denominator in
\cref{eq:fixed_points1} and a suitable large deviation bound on Bernoulli
variables. Details are standard and is left to the reader.
\end{proof}
\begin{proof}[Proof of \cref{cor:conditionA}]
Follows from \cref{lem:moment_control,lem:fixed_points,lem:largest_cycle}.
\end{proof}
\section{Proofs of the lemmas in \cref{sec:Galton_Watson}}\label{sec:appendix2}

 In this section, we shall prove \cref{lem:lowerdev,lem:max_ring}. 

\medskip

Let $p_i$ denote $\P(\xi = i)$ for $i \in \N$ and denote the generating
function by $\varphi(s) = \sum_ip_is^i $. Let $\mu = \E \xi$. Let
$Z_n$ denote the number of offsprings in the $n$-th generation of the
Galton-Watson process

\subsection{Critical Galton-Watson trees}
We assume $\xi$ has geometric distribution with parameter $1/2$. Here $\mu = 1$ and we want to show that $Z_r$ cannot be much more than $r$. The
following large deviation result is a special case of the main theorem of
\cite{ldev}.

\begin{prop}\label{prop:ldev_critical}
For all $r \ge 1$ and $k \ge 1$, 
\[
 \P(Z_r \ge k)< \frac{3}{2}\left(1+\frac{1}{\varphi''(3/2)r/2+2}\right)^{-k}
\]
\end{prop}

\subsection{Supercritical Galton-Watson trees}
Here $\mu>1$. Recall the assumptions 
\begin{itemize}
 \item $0<p_0+p_1<1$
\item There exists a small enough $\lambda>0$ such that $\E(e^{\lambda \xi}) <
\infty$.
\end{itemize}
It is well known (see \cite{Harris}) that
$Z_n/\mu^n$ is a martingale which converges almost surely to some non-denegerate
random variable $W$. Let $\rho : = P(\lim_n Z_n = 0)$ be the extinction
probability which is strictly less than $1$ in the supercritical
regime. 

The following results may be
realized as special cases of the results in \cite{Athreya2}, \cite{lowerdev} and
further necessary references can be found in these papers.  

$W$ if restricted to $(0,\infty)$ has a strictly positive continuous density
which is denoted by $w$. In other words, we have the following limit theorem:
\[
 \lim_n\P(Z_n \ge x\mu^n) = \int_{x}^{\infty}w(t)dt,\quad \quad x>0
\]
Also define $\gamma:=\varphi'(\rho)$ where $0<\gamma<1$ in our case. Define
$\beta$ by the relation $\gamma = \mu^{-\beta}$. It is clear that in our case
$\beta \in (0,\infty)$. $\beta$ is used to determine the behaviour of $w$ as $x
\downarrow 0$.  The following is proved in \cite{Athreya2}.
\begin{prop}
 Let $\eta:=\mu^{\beta/(3+\beta)} >1$. Then for all $\varepsilon  \in (0,\eta)$,
there exists a positive constant $C_{\varepsilon} > 0$ such that for all $k \ge
1$, 
\begin{equation}
 \lvert \P(Z_r = k)\mu^r - w(k/\mu^r)\rvert \le
C_{\varepsilon}\frac{\eta^{-r}}{k\mu^{-r}}+(\eta-\varepsilon)^{-r}\label{eq:lowerdev1}
\end{equation}
for all $r\ge1$.
\end{prop}
It can be shown (see \cite{Biggins}) that there exists positive constants $A_1 > 0,
A_2 >0$ such that $A_1x^{\beta-1}<w(x)<A_2x^{\beta-1}$ as $x \downarrow 0$.
Using this and \cref{eq:lowerdev1}, we get
\begin{equation}
 \P(Z_r = k) \le C\frac{k^{\beta-1}}{\mu^{r\beta}} + \frac{\eta^{-r}}{k} +
((\eta-\varepsilon)\mu)^{-r}\label{eq:lowerdev2}
\end{equation}

\begin{proof}[Proof of \cref{lem:lowerdev}]
 The proof is straightforward by summing $k$ from $1$ to $\gamma^r$ the
expression given by the right hand side of \eqref{eq:lowerdev2}.
\end{proof}
\subsection{Random plane trees}

\begin{proof}[Proof of \cref{lem:max_ring}]
  Note that it is enough to prove the bound for $r \le n$ because otherwise the probability is $0$.
  It is well known that if we pick an oriented edge uniformly from $U_{0,n}$ and
re-root the tree there then the distribution of this new re-rooted tree is the
same as that of $U_{0,n}$ (see \cite{treeprofile}). Let $V$ denote the root
vertex of the new re-rooted tree and let $\mathcal Z_j(V)$ denote the
number of vertices at distance exactly $j$ from $V$. It is well known
that the probability of a critical geometric Galton-Watson tree to have $n$
edges is $\asymp n^{-3/2}$. Using this fact and 
\cref{prop:ldev_critical} we get for any $k \ge 1$ and $1\le j \le  r$
\begin{equation}
\P(\mathcal Z_j(V)>k) < n^{3/2}c\exp(-c'k/j)<n^{3/2}c\exp(-c'k/r)\label{eq:max_ring1}
\end{equation}
and some suitable positive constants
$c,c'$. Note that if $M_r >
r^2\log^2n$ then $\mathcal Z_j(v)>r\log^2n$ for some $1\le j \le r$
and some vertex $v \in U_{0,n}$. Using this and union bound to the
estimate obtained in \eqref{eq:max_ring1}, we get
 \begin{equation}
\P(M_r > r^2\log^2n) <cn^{5/2}r\exp(-c'\log^2n) = O(\exp(-c'\log^2n))\nonumber
\end{equation}
for some positive constants $c$ and $c'$. This completes the proof.
 \end{proof}
\end{appendix}

\bigskip
\noindent
{\sc Gourab Ray}, {\em UBC}, {\tt <gourab1987@gmail.com>}

\end{document}